\documentclass[11pt]{article}
\usepackage{color}

\author{}
\date{}

\textheight 23cm\begin{large}\begin{normalsize}                             \end{normalsize}                \end{large}
\usepackage{bbm,amsmath,amsthm,amsfonts,amssymb}

\usepackage{graphicx}

\newcommand{\captionfonts}{\footnotesize}
\makeatletter  % Allow the use of @ in command names
\long\def\@makecaption#1#2{%
  \vskip\abovecaptionskip
  \sbox\@tempboxa{{\captionfonts #1: #2}}%
  \ifdim \wd\@tempboxa >\hsize
    {\captionfonts #1: #2\par}
  \else
    \hbox to\hsize{\hfil\box\@tempboxa\hfil}%
  \fi
  \vskip\belowcaptionskip}
\makeatother   % Cancel the effect of \makeatletter
\newcommand{\nwc}{\newcommand}

\newtheorem{proposition}{Proposition}[section]
\newtheorem{lemma}[proposition]{Lemma}
\newtheorem{remark}[proposition]{Remark}
\newtheorem{theorem}[proposition]{Theorem}
\newtheorem{corollary}[proposition]{Corollary}
\newtheorem{definition}[proposition]{Definition}

\nwc{\R}{\mathbb R}
\nwc{\Z}{\mathbb Z}
\nwc{\N}{\mathbb N}

\newcommand{\ignore}[1]{}

\nwc{\eps}{\varepsilon}
\nwc{\re}{Re\,}

\nwc{\wto}{\rightharpoonup}

\nwc{\ds}{\displaystyle}
\newcommand {\bedis} {\begin{displaymath}}
\newcommand {\edis} {\end{displaymath}}
\newcommand{\newbeqna} {\renewcommand {\arraystretch} {2}
                        \begin {displaymath} \begin {array}{crcl}}
\newcommand{\neweqna}{\end{array} \end {displaymath}}

\newcommand{\fbeqna}{\renewcommand {\arraystretch} {1.3}
\begin {displaymath}\begin{array}{rcll}}
\newcommand{\feqna}{\end{array}\end{displaymath}}

\newcommand {\beqna} {\begin{eqnarray*}}
\newcommand {\eqna} {\end{eqnarray*}}
\newcommand {\beqn} {\begin{eqnarray}}
\newcommand {\eqn} {\end{eqnarray}}

\begin{document}
\title{Self-similar solutions with fat tails for Smoluchowski's coagulation equation with
locally bounded kernels}

\author{%
B. Niethammer%
\footnote{Institute of Applied Mathematics,
University of Bonn, Endenicher Allee 60, 53115 Bonn, Germany}
and
J. J. L. Vel\'{a}zquez\footnote{Institute of Applied Mathematics,
University of Bonn, Endenicher Allee 60, 53115 Bonn, Germany}
}
\maketitle

\begin{abstract}
The existence of self-similar solutions with fat tails for Smoluchowski's coagulation equation has so far only been established for the solvable
kernels and the  diagonal one. In this paper we prove the existence of such self-similar solutions for continuous kernels $K$ that are homogeneous
of degree $\gamma \in [0,1)$ and satisfy $K(x,y) \leq C (x^{\gamma} + y^{\gamma})$. More precisely, for any $\rho \in (\gamma,1)$ we establish
the existence of a continuous weak nonnegative self-similar profile with decay $x^{-(1{+}\rho)}$ as $x \to \infty$.

For the proof we consider the time-dependent problem in self-similar variables with the aim to use a variant of Tykonov's fixed point theorem 
to establish the existence of a stationary profile. This requires to identify a weakly compact subset that is invariant under the evolution.
In our case we define a set of nonnegative measures which encodes the desired decay behaviour in an integrated form. The main difficulty is 
to establish the invariance of the lower bound under the evolution. Our key idea is to choose as a test function in the time dependent problem 
the solution of the associated backward dual problem. 
\end{abstract}
{\bf AMS subject class:} 45K05;  82C05


\section{Introduction}

\subsection{Smoluchowski's coagulation equation and self-similarity}

In this paper we investigate the existence of so-called fat-tail self-similar solutions for the classical  coagulation equation 
by Smoluchowski  \cite{Smolu16} 
that describes irreversible aggregation of clusters through binary collisions. If $f(\xi,t)$ denotes the density
of clusters of mass $\xi$ at time $t$, the evolution of $f$ is given by  
\begin{equation}\label{smolu1}
 \partial_t f (\xi,t) =\frac{1}{2} \int_0^{\xi} K(\xi-\eta,\eta) f(\xi{-}\eta,t) f(\eta,t)\,d\eta - f(\xi,t) \int_0^{\infty}
K(\xi,\eta)f(\eta,t)\,d\eta\,,
\end{equation}
where the rate kernel $K$ describes the  rate of coagulation of  clusters of size $\xi$ and $\eta$. 
This model is used in a wide variety of applications, most notably in the kinetics of polymerization and aerosol physics, but also
in astrophysics and mathematical biology, for example. We refer to \cite{Aldous99,Drake72,LauMisch04} for
further background on applications of \eqref{smolu1} and its mathematics. 
 
In the following we consider homogeneous kernels with degree $\gamma \in [0,1)$. It is well-known 
for a large class of kernels that in this case  for data with
finite first moment, the model \eqref{smolu1} is well-posed and preserves the first moment for all times. Well-posedness of the model
for data with possibly infinite first moment but finite $\gamma$-th moment 
has also been established for a range of kernels \cite{FouLau06b}. 

A fundamental issue in the theory of coagulation is the so-called scaling hypothesis that states that for homogeneous kernels solutions approach
a unique self-similar profile for large times. Despite a significant range of results based on formal asymptotics (see in particular \cite{Leyvraz03,vanDoErnst88})
mathematically rigorous results supporting this hypothesis are still rare except for the special case of solvable kernels,
that is $K(x,y)=2, K(x,y)=x+y$ and $K(x,y)=xy$. 
Self-similar solutions for kernels of homogeneity $\gamma<1$ are of the form
\begin{equation}\label{ss1}
 f(\xi,t) = \frac{1}{t^{\alpha}} g\Big( \frac{\xi}{t^{\beta}}\Big)\,, \qquad \alpha = 1+(1{+}\gamma) \beta\,,
\end{equation}
where the self-similar profile  $g$ solves
\begin{equation}\label{ss2}
 -(1+(1{+}\gamma)\beta) g - \beta x g'(x) =\frac{1}{2} \int_0^{x} K(x,y) g(x{-}y) g(y)\,d\eta - g(x) \int_0^{\infty}
K(x,y)g(y,t)\,d\eta\,.
\end{equation}
Since for some kernels one cannot necessarily expect that the integrals on the right-hand side are finite, it is convenient to rewrite
the equation. In fact, multiplying the equation by $x$ and rearranging,  a weak formulation of \eqref{ss2} is that $g$ solves
\begin{equation}\label{ss3}
 \beta \partial_x ( x^2 g(x) )= \partial_x \Big[\int_0^x \int_{x{-}y}^{\infty} y K(y,z) g(z) g(y)\,dz\,dy \Big] + ((1{-}\gamma) \beta -1 )x g(x)\,
\end{equation}
in a distributional sense. 
If one in addition requires that the solution has finite first moment, then this  also fixes $\beta = 1/(1{-}\gamma)$ and in this case the second
term on the right hand side of \eqref{ss3} vanishes.

Let us now first describe what is known about self-similar solutions and the scaling hypothesis for the 
constant kernel which is the only solvable one with homogeneity 
smaller than 1. In this case it is easily checked that there is an explicit self-similar solution with finite mass,
given by the self-similar profile $g(x)=e^{-x}$. Convergence to this self-similar solution has been established under some assumptions on the initial
data in several papers \cite{DaCosta96,DMR00,KreerPen94,LauMisch05}.  A complete characterization of its domain of attraction has more recently 
been given  in 
\cite{MePe04}. Moreover, it is also proved in \cite{MePe04}  that there exists  a family of self-similar
solutions with infinite mass, so-called self-similar solutions with fat tails. More precisely, it was established that for any $\rho \in (0,1)$
there exists a self-similar profile with decay $x^{-(1+\rho)}$. Furthermore it is shown that a solution to the coagulation equation converges to the self-similar solution
with decay behaviour $x^{-(1+\rho)}$ if and only if the integrated mass distribution is regularly varying with exponent $1-\rho$. The proof is simple and elegant,
but relies on the use of the Laplace transform and hence the methods are not applicable to non-solvable kernels.

In fact, for non-solvable kernels significantly less is known about the scaling hypothesis. Only rather recently results on the existence of 
self-similar profiles have become available \cite{EMR05,FouLau05} and certain properties of these profiles have been established 
\cite{CanMisch11,EsMisch06,FouLau06a}. However,
until now  their domains of attraction under the evolution \eqref{smolu1} are completely unknown and related to this, uniqueness of self-similar
profiles (in a certain class, e.g. with finite mass) is still an open question. 

Furthermore, also the existence of self-similar solutions with fat tails has not been established for non-solvable kernels apart from the
diagonal one  \cite{NV11a}.
 It is the goal of the present paper to show the existence of self-similar profiles with fat tails for kernels of 
homogeneity $\gamma$ that are  bounded by $C(x^{\gamma}+y^{\gamma})$. As we point out in Remark \ref{R.remark}
this covers a wide range of kernels considered in the literature, but does not address some other kernels of interest, such as Smoluchowski's
classical kernel $(x^{1/3} + y^{1/3})(x^{-1/3} + y^{-1/3})$.

\subsection{The main result}

In order to present our main results and the ideas of the proof we go over to the  mass density function
$h\left(  x,t\right)  =xg\left(  x,t\right)  $ and introduce the parameter $\rho=\gamma+\frac{1}{\beta}$. Then, 
after rescaling,  the time dependent version of equation \eqref{ss3} becomes
\begin{equation}
\partial_{t}h+\partial_{x}\left[  \int_{0}^{x}\int_{x-y}^{\infty}%
\frac{K\left(  y,z\right)  }{z}h\left(  z\right)  h\left(  y\right)
\,dz\,dy\right]  -\beta\left[  \partial_{x}\left(  xh\right)  +\left(  \rho{-}1\right)
h\right]  =0\,, \label{A2}%
\end{equation}
with initial data
\begin{equation}\label{A3}
 h(x,0)=h_0(x)\,.
\end{equation}

Our precise assumptions on the kernel $K$ are as follows. We assume that $K$ satisfies
\begin{equation}\label{Ass1a}
 K \in C^0 ((0,\infty) \times (0,\infty))\,, \qquad K(x,y)=K(y,x)\geq 0 \qquad \mbox{ for all } x,y \in (0,\infty) \,,
\end{equation}
is homogeneous of degree $\gamma \in [0,1)$, that is
\begin{equation}\label{Ass1b}
 K(ax,ay)=a^{\gamma} K(x,y) \qquad \mbox{ for all } x,y \in (0,\infty)\,,
\end{equation}
 and satisfies the growth condition
\begin{equation}\label{Ass1}
K(x,y)  \leq C\left( x^{\gamma} +  y^{\gamma} \right)\,\qquad \mbox{ for all } x,y \in (0,\infty)\,.
\end{equation}

Our main result can now be formulated as follows

\begin{theorem}\label{T.main}
 Given $\gamma \in [0,1)$ and a kernel $K$ that satisfies assumptions \eqref{Ass1a}-\eqref{Ass1},  then for any $\rho \in (\gamma,1)$
 there exists
a weak stationary solution $h$ to \eqref{A2}. This solution is nonnegative, continuous and satisfies
\[
h(x) \sim (1{-}\rho) x^{-\rho} \, \mbox{ as } x \to \infty.
\]
\end{theorem}

\begin{remark}\label{R.remark}
{\em 
 While our result covers a wide range of kernels, in particular for example the product kernel $K(x,y)=(xy)^{\gamma/2}$ and the general sum kernel
$K(x,y) = x^{\alpha} y^{\gamma-\alpha} + y^{\alpha} x^{\gamma-\alpha}$ with $\alpha \geq 0$, it does not apply to singular kernels, such as
Smolochowski's
kernel $K(x,y) = (x^{1/3} + y^{1/3}) ( x^{-1/3} + y^{-1/3})$. 
The main reason is that for the type of kernels considered in this paper our global estimate \eqref{B1a} suffices to  prove
 that the nonlinear integral terms in \eqref{A2} are well defined, but this is not  sufficient for singular kernels.
However, we expect that we can use regularizing properties of singular kernels to show that $h$ is small as $x \to 0$. Then  it is possible
to extend  our strategy to cover also this range of kernels with some additional technical effort. This analysis will be the subject of 
future work.}
\end{remark}

\begin{remark}\label{R.remark2}
{\em We prove under rather minimal assumptions on the kernel $K$ the existence of a weak continuous solution. One would expect that for a kernel
that is locally smooth, the self-similar solution is locally smooth as well. The proof of such a property is however not, as one might first 
expect, a straightforward bootstrap
argument  due to the possibly singular behavior of solutions near $x=0$ and we do not further explore this issue in this paper.

A related issue is the behaviour of self-similar profiles as $x \to 0$.
For mass-conserving self-similar solutions for
product type kernels with $\gamma \in (0,1)$ it has for example
 been rigorously established in \cite{NV11b}, that solutions behave as $h(x) \sim
c x^{-\gamma}$ as $x \to 0$ (see also \cite{EsMisch06,FouLau06a} for related results on other type of kernels). However, without any further assumptions
on the kernel, such as certain lower bounds, 
 we cannot expect a universal behaviour as $x \to 0$. This already follows from the trivial observation
that our result also applies to the case $K \equiv 0$ for which $h(x)=(1{-}\rho)x^{-\rho}$. 
}
\end{remark}

\subsection{Strategy of proof}

Our strategy to find a stationary solution to \eqref{A2} will in principle be the following. We consider the corresponding evolution problem
and prove that it preserves a convex set that is compact in the weak topology and contains functions with the expected decay behaviour. 
This will allow us to prove the existence of a fixed point by a variant of Tykonov's fixed point theorem. 
However, it is not so easy to prove well-posedness directly for \eqref{A2}-\eqref{A3} since we need to consider the well-posedness of the problem
in a space of functions that are singular at the origin. Uniqueness and continuous dependence is
difficult to prove for (\ref{A2})-(\ref{A3}) without careful asymptotic
estimates for the solutions near the origin.
Instead, we will consider a family of regularized problems and prove that self-similar solutions
for this regularized problem exist and satisfy uniform estimates that allow us to pass to the limit in the corresponding equation. 

We now describe  the regularization procedure in more detail. We consider  a family of  problems
\begin{align}
\partial_{t}h+\partial_{x}\left[  \int_{0}^{x}\int_{x-y}^{\infty}%
\frac{K_{\lambda}\left(  y,z\right)  }{z}h\left(  z\right)  h\left(
y\right) \,dz\,dy \right]  -\beta\left[  \partial_{x}\left(  xh\right)  +\left(
\rho{-}1\right)  h\right]   &  =0\ ,\ \lambda>0\,,\label{A2a}\\
h\left(  x,0\right)   &  =h_{0}\left(  x\right)\,,  \label{A3b}%
\end{align}
where we  define $K_{\lambda}$  by means of
\begin{equation}
K_{\lambda}\left(  y,z\right)  =K\left(  y,z\right)  \zeta\left(  \frac
{y}{\lambda}\right)  \zeta\left(  \frac{z}{\lambda}\right)  \zeta\left(
\frac{y}{\lambda\left(  y+z\right)  }\right)  \zeta\left(  \frac{z}%
{\lambda\left(  y+z\right)  }\right)  \  ,\ \ 0<\lambda<\frac{1}{2}\,,
\label{D1}%
\end{equation}
where $\zeta\in C^{\infty}\left[  0,\infty\right)  $ is a cutoff function
satisfying $\zeta^{\prime}\geq0,\ \zeta\left(  s\right)  =0$ if $s\leq\frac
{1}{2},\ \zeta\left(  s\right)  =1$ if $s\geq1.$

We will obtain existence and uniqueness of solutions to  the problem (\ref{A2a})-(\ref{A3b}) using standard
fixed point arguments (cf. Proposition \ref{WPR} and Lemma \ref{L2}) and prove continuity of the corresponding
semi-group in the weak topology (Proposition \ref{weakcontin}). This can be done in suitable subsets of certain Banach spaces.
 More precisely,
 we consider the metric space $\mathcal{X}_{\rho}$  of nonnegative Radon
measures, which we denote by some abuse of notation by  $h dx\in\mathcal{M}^{+}\left(  \left[  0,\infty\right)  \right)  $,
satisfying the condition
\begin{equation}
\sup_{R\geq0}\frac{\int_0^R h\left(  x\right)  dx}%
{R^{1{-}\rho}}<\infty\,.\label{C1}%
\end{equation}
Since $h dx$ might contain Dirac masses away from the origin, we need to make the notation $\int_0^R h(x)\,dx$ precise. Here and throughout the paper
we understand this integral in the sense of $\int_0^R h(x)\,dx = \int_{[0,R]}h(x)\,dx$.

We give $\mathcal{X}_{\rho}$ the structure of a metric space by means of
\begin{equation}
\left\Vert h\right\Vert _{\mathcal{X}_{\rho}}    =\sup_{R\geq0}\frac
{\big | \int_0^R h\left(  x\right)  dx \big | }{R^{1{-}\rho}}\qquad \mbox{ and } \qquad 
\operatorname*{dist}\left(  h_{1},h_{2}\right)     =\left\Vert h_{1}%
-h_{2}\right\Vert _{\mathcal{X}_{\rho}}\,. \label{C1b}%
\end{equation}

Given any $T>0$ we can define a metric space $C\left(  \left[  0,T\right]
;\mathcal{X}_{\rho}\right)  $ via
\begin{equation}
\left\Vert h\right\Vert =\sup_{0\leq t\leq T}\left\Vert h\right\Vert
_{\mathcal{X}_{\rho}}\ 
 ,\ \qquad \operatorname*{dist}\left(  h_{1},h_{2}\right)
=\left\Vert h_{1}-h_{2}\right\Vert\,. \label{C7}%
\end{equation}
The set that will be shown to be invariant under the evolution induced by
(\ref{A2a})-(\ref{A3b}) will be the set $\mathcal{Y}$  of measures $h\in\mathcal{M}%
^{+}\left(  \left[  0,\infty\right)  \right)  $ satisfying
\begin{align}
\int_0^R h\left(  x\right)  dx  &  \leq 
R^{1{-}\rho
}\, ,\qquad R\geq0\label{B1a}\\
\int_0^R h\left(  x\right)  dx  &  \geq R^{1{-}\rho}\left(
1-\frac{R_{0}^{\delta}}{R^{\delta}}\right)  _{+}\,, \qquad R\geq
0\,,\label{B2a}%
\end{align}
for a sufficiently large $R_0$ and a sufficiently small $\delta>0$.
It is straightforward to see that this set is convex and compact in the  weak topology.

The heart of our analysis is the proof of the  invariance of \eqref{B1a} and \eqref{B2a} under the evolution 
(\ref{A2a})-(\ref{A3b}). The upper bound \eqref{B1a} can be proved by analyzing a simple differential inequality 
that is satisfied
by $\int_{\left[  0,R\right]  }h\left(  x\right)  dx$ (cf. Proposition \ref{P.upperbound}). The proof of the invariance of 
(\ref{B2a}) is 
more delicate and is contained in Sections \ref{S.massflux}-\ref{S.lowerbound}. 

To explain the main idea it is useful to comment first on the particular choice \eqref{B2a} for a lower bound. If we consider \eqref{A2a}-\eqref{A3b} with 
$K_{\lambda} \equiv 0$ we obtain a linear transport equation that has the explicit solution $h(x,t) =e^{\rho \beta t} h_0\big( x e^{\beta t}\big)$.
If we take data $h_0(x)= x^{-\rho} \big( 1-Cx^{-\delta}\big)$ we find
\[
 h(x,t) = x^{-\rho} \Big( 1 - \frac{C e^{-\delta \beta t}}{x^{\delta}}\Big) \sim x^{-\rho} \Big( 1- \frac{C}{x^{\delta}}\Big) + C \delta \beta
t x^{-(\rho+\delta)}
\]
and thus obtain an improved lower bound for positive times. Our task is then to show that the additional error terms induced by the nonlinear
coagulation term can be absorbed into the positive term if $\delta$ is sufficiently small. 

In order to prove the invariance of \eqref{B2a} we now write \eqref{A2a}-\eqref{A3b} as 
\begin{align}
\partial_{t}h\left(  x,t\right) & +h\left(  x,t\right)  \int_{0}^{\infty}\frac{K_{\lambda
}\left(  x,z\right)  }{z}h\left(  z,t\right)\,dz
\nonumber\\
&  -\int_{0}^{x}\frac{K_{\lambda}\left(  y,x-y\right)  }{\left(  x-y\right)
}h\left(  x-y,t\right)  h\left(  y,t\right)\,dy -\beta\left[  x\partial_{x}h+\rho h\right]     =0\label{D2}\\
h\left(  0,\cdot\right)   &  =h_{0}\,. \label{D3}%
\end{align}
It is easy to see, by  testing with a function $\psi=\psi(x,t)$, that one obtains 
\[
 \int h(x,t) \psi(x,t)\,dx = \int h_0(x) \psi(x,0)\,dx\,
\]
if $\psi$ solves the associated dual problem 
\begin{equation}
-\partial_{s}\psi\left(  x,s\right)  -\int_{0}^{\infty}\frac{K_{\lambda
}\left(  x,z\right)  }{z}h\left(  z,s\right)  \left[  \psi\left(
x{+}z,s\right)  -\psi\left(  x,s\right)  \right]\,dz  +\beta x\partial_{x}%
\psi\left(  x,s\right)  -\beta\left(  \rho{-}1\right)  \psi\left(  x,s\right)
=0 \label{S7a}
\end{equation}
with 
\begin{equation}\label{S7b}
 \psi(x,t)= \chi_{[0,R]}(x)\,.
\end{equation}
Thus, in order to estimate $\int_0^R h(x,t)\,dx$ we need to estimate $\psi(x,0)$ from below. It is worth remarking here that also our proof
of weak continuity of the semi-group relies on using the solution of the dual problem as a test function. 
With regard to an estimate of $\psi(x,0)$ as above, it turns out that in the case of kernels satisfying 
\eqref{Ass1} and measures $h$ satisfying \eqref{B1a} and \eqref{B2a} we can construct a subsolution for $\psi$ by replacing the term
$\frac{K_{\lambda
}\left(  x,z\right)  }{z}h\left(  z,s\right) $
by a suitable power law (Lemma \ref{L.comparison}).
The equation for the subsolution has an explicit self-similar solution (Proposition \ref{P.W}).  Finally, it remains to work out that this subsolution is sufficiently good to show that \eqref{B2a} is preserved under the
evolution. This is done in Section \ref{S.lowerbound}, see Proposition \ref{P.lowerbound}. 

A variant of Tykonov's fixed point theorem now guarantees the existence of a stationary solution $h_{\lambda}$ to \eqref{A2a}. 
The invariance of $\mathcal{Y}$ together with its weak compactness allows us to find a subsequence that converges weakly to a measure $h$. 
Since $h \in \mathcal{Y}$ it is not difficult to show that $h$ is also a weak stationary solution to \eqref{A2} (cf. Proposition \ref{P.main}).
In the last two subsections we then show that this weak solution is in fact also continuous on $(0,\infty)$ (cf. Lemma \ref{L.cont}) and has the
desired decay behaviour (cf. Lemma \ref{h.tail}).

\section{Analysis of the regularized problems (\ref{A2a})-(\ref{A3b})}\label{S.regularized}

\subsection{Well posedness of the regularized problem.}

As a first step we  study the regularized equation \eqref{D2}-\eqref{D3} for which well-posedness can be easily proved.
 This will allow us to
define a family of evolution semigroups $\left\{  S_{\lambda}\left(  t\right)
\right\}  _{\lambda>0}.$ 

We introduce the change of variables
\begin{equation}
x=Xe^{-\beta t}\ \ ,\ \ h\left(  x,t\right)  =H\left(  X,t\right)
 \label{D6}%
\end{equation}
such that  (\ref{D2})-(\ref{D3}) becomes 
\begin{align}
 \partial_{t}H\left(  X,t\right) & +H\left(  X,t\right)  \int_{0}^{\infty
}\frac{K_{\lambda}\left(  Xe^{-\beta t},Ze^{-\beta t}\right)  }{Z}H\left(
Z,t\right)\,dZ  \nonumber\\
&-\int_{0}^{X}\frac{K_{\lambda}\left(  Ye^{-\beta t},\left(
X{-}Y\right)  e^{-\beta t}\right)  }{\left(  X{-}Y\right)  }H\left(  X{-}Y,t\right)
H\left(  Y,t\right) \,dY 
-\beta\rho H\left(  X,t\right)     =0\label{D4}\\
H\left(  0,\cdot\right)   &  =h_{0}\ \ \label{D5}%
\end{align}
This can be rewritten as 
\begin{equation}
\partial_{t}H\left(  X,t\right)  +\mathcal{A}\left[  H\right]  \left(
X,t\right)  H\left(  X,t\right)  -\mathcal{Q}\left[  H\right]  \left(
X,t\right)  =0\ \label{D7a}\,,
\end{equation}
with 
\begin{align}
\mathcal{A}\left[  H\right]  \left(  X,t\right)   &  =\int_{0}^{\infty}%
\frac{K_{\lambda}\left(  Xe^{-\beta t},Ze^{-\beta t}\right)  }{Z}H\left(
Z,t\right)  dZ-\beta\rho\label{D7b}\,,\\
\mathcal{Q}\left[  H\right]  \left(  X,t\right)   &  =\int_{0}^{X}%
\frac{K_{\lambda}\left(  Ye^{-\beta t},\left(  X{-}Y\right)  e^{-\beta
t}\right)  }{\left(  X{-}Y\right)  }H\left(  X{-}Y,t\right)  H\left(  Y,t\right)\,dY\,.
\label{D7c}%
\end{align}
This particular reformulation is convenient in order to preserve the
nonnegativity of $H$ in fixed point arguments.

\begin{definition}
\label{def1}We say that $H \in C\left(  \left[  0,T\right]  ;\mathcal{X}_{\rho}\right)  $
is a mild solution
of (\ref{D4})-(\ref{D5}) if it satisfies for every $t \in [0,T]$ the equation
\begin{equation}
\begin{split}
H\left(  X,t\right) 
= & h_{0}\left(  X\right)  \exp\left(  -\int_{0}^{t}\mathcal{A}\left[  H\right]
\left(  X,s\right)  ds\right) \\
& \quad  +\int_{0}^{t}\exp\left(  -\int_{s}%
^{t}\mathcal{A}\left[  H\right]  \left(  X,\xi\right)  d\xi\right)
\mathcal{Q}\left[  H\right]  \left(  X,s\right)\,ds =: \mathcal{T}[H](X,t)   \label{D8a}%
\end{split}
\end{equation}
in the sense of measures. 
\end{definition}
Note that  if $H(\cdot,t)$ is a measure and $K$ is continuous,
then $\mathcal{A}$ is continuous in $X$. Furthermore, $Q$ is a weighted convolution of measures and hence also a measure.
Thus,  $\mathcal{T}$ is well-defined.

\begin{lemma}
\label{L.WPR}
 For any $\lambda>0$ and any $h_{0}\in\mathcal{X}_{\rho}$
there exists a time $T>0$ and a unique mild solution of (\ref{D4})-(\ref{D5}) in $[0,T]$.
\end{lemma}
\begin{proof}
We are going to prove the well-posedness  of (\ref{D8a}) in the metric space $C\left(
\left[  0,T\right]  ;\mathcal{X}_{\rho}\right)  $ using a fixed point argument. In the following it will be crucial
that $K_{\lambda} (X,Z)=0$ for small $X,Z$. As a consequence all the constants will depend on $\lambda$.

Our goal is to prove that  the operator $\mathcal{T}  $ maps
 the subset
\[
\mathcal{U}=\left\{  h\in C\left(  \left[  0,T\right]  ;\mathcal{X}_{\rho
}\right)  :\left\Vert h\right\Vert \leq2\left\Vert h_{0}\right\Vert
_{\mathcal{X}_{\rho}}\right\} \subset C\left(  \left[  0,T\right]  ;\mathcal{X}_{\rho}\right)
\]
into itself and is strongly contractive if $T=T(\lambda)$ is sufficiently small.

In the following we will often use that $\|h\|_{\mathcal{X}_{\rho}}\leq C_0$ implies that
\begin{equation}\label{dyadic}
 \int_{x}^{\infty} \frac{h(z)}{z^{\alpha}}\,dz \leq C_0 C(\alpha) x^{1-\rho-\alpha} \qquad \mbox{ if } \alpha >1-\rho.
\end{equation}
 In fact, using a dyadic decomposition, we find
\[
 \begin{split}
  \int_x^{\infty} \frac{h(z)}{z^{\alpha}}\,dz & = \sum_{n=0}^{\infty} \int_{2^nx}^{2^{n+1}x} 
\frac{h(z)}{z^{\alpha}}\,dz\\
&\leq \sum_{n=0}^{\infty}  \big( 2^n x\big)^{-\alpha} \int_{2^nx}^{2^{n+1}x}  h(z)\,dz\\
& \leq C_0  \sum_{n=0}^{\infty}  \big( 2^n x\big)^{-\alpha} \big(2^{n+1}x\big)^{1-\rho} = 2 C_0 x^{1-\rho-\alpha} 
\sum_{n=1}^{\infty} 2^{n(1-\rho-\alpha)}
 \end{split}
\]
and thus \eqref{dyadic} follows.

We first estimate $\mathcal{A}\left[  H\right]  $ for
$H\in\mathcal{X}_{\rho}.$ Using that $K_{\lambda}\left(  Xe^{-\beta
t},Ze^{-\beta t}\right)  $ vanishes if $Z\leq\lambda$ and \eqref{dyadic} we obtain that
$\mathcal{A}\left[  H\right]  \left(  X,t\right)  \leq C_{\lambda}$ for
$H\in\mathcal{U}$,\ $X\leq1.$ On the other hand, in order to estimate
$\mathcal{A}\left[  H\right]  \left(  X,t\right)  $\ for $X\geq1$ we use the
fact that $K_{\lambda}\left(  Xe^{-\beta t},Ze^{-\beta t}\right)  $ vanishes
if $\left(  1-\frac{\lambda}{2}\right)  Z\leq\frac{\lambda}{2}X.$ Then \eqref{dyadic}
implies that
\[
\int_{0}^{\infty}\frac{K_{\lambda}\left(  Xe^{-\beta t},Ze^{-\beta t}\right)
}{Z}H\left(  Z,t\right)  dZ\leq C \int_{\frac{\lambda}{2-\lambda}X}^{\infty}
\frac{X^{\gamma} + Z^{\gamma}
}{Z}H(Z,t)dZ\leq C_{\lambda} X^{\gamma-\rho}\,, \qquad X \geq 1\,.
\]
Since $\rho>\gamma$ it  follows that $\mathcal{A}\left[  H\right]  \left(
X,t\right)  \leq C_{\lambda}$ for $H\in\mathcal{U}$,\ $X\geq1.$ Therefore:%
\begin{equation}
\mathcal{A}\left[  H\right]  \left(  X,t\right)  \leq C_{\lambda
}\ \ \text{for\ \ }H\in\mathcal{U\ },\text{\ \ }X>0 \,.\label{D8b}%
\end{equation}

Moreover, we can  estimate $\mathcal{Q}\left[  H\right]  \left(
X,t\right)  $ in the norm $\left\Vert \cdot\right\Vert _{\mathcal{X}_{\rho}}$.
In fact, due to \eqref{Ass1} and \eqref{D1} we find
\begin{align*}
\int_0^R\mathcal{Q}\left[  H\right]  \left(  X,t\right)
dX  
&  \leq C \int_0^R\int_{\max(\lambda/2, X \lambda/2)}^{X(1-\lambda/2)} \frac{Y^{\gamma} + (X{-}Y)^{\gamma}}{(X{-}Y)}
H\left(  X{-}Y,t\right)  H\left(
Y,t\right)\,dY\,dX \\
&\leq C \int_{\lambda/2}^R \int_{Y\lambda/2}^R
 \frac{Y^{\gamma} + X^{\gamma}}{X} H(X)\,dX \,H(Y)\,dY\,.
\end{align*}
If $R \leq 1$ the above estimate implies that 
\[
\int_{0}^R \mathcal{Q}[H](X,t)\,dX \leq C_{\lambda} R^{\gamma} \left(  R^{1{-}\rho}\right)
^{2}\leq C_{\lambda}R^{1-\rho}\,.
\]
To treat the case $R>1$, recall  that due to $H(\cdot,t) \in \mathcal{X}_{\rho}$, estimate \eqref{dyadic} and the fact
that $\gamma< \rho$  we have that
\[ \int_{Y\lambda/2}^R \frac{H(X)}{X}\,dX \leq C_{\lambda} Y^{-\rho} \qquad \mbox{ and } \qquad 
\int_{\lambda^2/4}^{\infty}\frac{H(X)}{X^{1-\gamma}}\,dX\leq C_{\lambda}\,.
\]
This implies, since $\gamma-\rho<0$, that
\begin{align*}
 \int_0^R\mathcal{Q}\left[  H\right]  \left(  X,t\right)
dX  
& \leq C \Big( \int_{\lambda/2}^R Y^{\gamma} H(Y)\int_{Y\lambda/2}^R\frac{H(X)}{X}\,dX\,dY + 
\int_{\lambda/2}^R H(Y) \int_{\lambda^2/4}^R \frac{H(X)}{X^{1-\gamma}}\,dX\,dY\Big)\\
& \leq C_{\lambda} \Big( \int_{\lambda/2}^R Y^{\gamma-\rho} H(Y)\,dY +  R^{1-\rho}\Big)\\
 & \leq C_{\lambda} R^{1{-}\rho}\,
\end{align*}
and thus
\begin{equation}
\left\Vert \mathcal{Q}\left[  H\right]  \left(  \cdot,t\right)  \right\Vert
_{\mathcal{X}_{\rho}}\leq C_{\lambda} \,.\label{D8c}%
\end{equation}
Using (\ref{D8a}), (\ref{D8b}) and (\ref{D8c}) it follows that for $H \in \mathcal{U}$
\[
\left\Vert \mathcal{T}\left[  H\right]  \left(  \cdot,t\right)  \right\Vert
_{\mathcal{X}_{\rho}}\leq2\left\Vert h_{0}\right\Vert _{\mathcal{X}_{\rho}%
}\ ,\qquad \ 0\leq t\leq T\,,
\]
if $T=T(\lambda)$ is sufficiently small. Note that for any
$H\in\mathcal{U}$ we have $\mathcal{T}\left[  H\right]  \geq0$ by
construction. Therefore $\mathcal{T}$ maps $\mathcal{U}$ to $\mathcal{U}$ if
$T>0$ is sufficiently small. 

Analogous arguments yield
\begin{align*}
\left\vert \mathcal{A}\left[  H_{1}\right]  \left(  X,t\right)  -\mathcal{A}%
\left[  H_{2}\right]  \left(  X,t\right)  \right\vert  &  \leq C_{\lambda
}\left\Vert H_{1}-H_{2}\right\Vert _{\mathcal{X}_{\rho}}\,,\ \ \ H_{1}%
,\ H_{2}\in\mathcal{U}\,,\\
\left\Vert \mathcal{Q}\left[  H_{1}\right]  \left(  \cdot,t\right)
-\mathcal{Q}\left[  H_{2}\right]  \left(  \cdot,t\right)  \right\Vert
_{\mathcal{X}_{\rho}}  &  \leq C_{\lambda}\left\Vert H_{1}\left(
\cdot,t\right)  -H_{2}\left(  \cdot,t\right)  \right\Vert _{\mathcal{X}_{\rho
}}\, ,\ \ H_{1},\ H_{2}\in\mathcal{U}\,.
\end{align*}
As a consequence we obtain that $\mathcal{T}$ is strongly contractive 
if $T$ is sufficiently small and Banach's fixed point theorem implies 
 that there exists a unique solution of
the equation $H=\mathcal{T}\left[  H\right]  $ in $\mathcal{U}$ in $[0,T]$.
\end{proof}

We need to prove that a mild solution of (\ref{D4})-(\ref{D5}) is also a weak solution in the
following sense.

\begin{definition}
\label{def2} We will say that $H$ is a weak solution of (\ref{D4})-(\ref{D5})
in $\left[  0,\infty\right)  \times\left[  0,T\right]  $ if for any ${t}\in\left[  0,T\right]  $ 
and any test function $\psi\in C_{0}^{1}\left(
\left[  0,\infty\right)  \times\left[  0,{t}\right]  \right)  $ we have:%
\begin{equation}\label{S2}
 \begin{split}
  \int & H\left(  X,{t}\right)  \psi\left(  X,{t}\right)  dX-\int
h_{0}\left(  X\right)  \psi\left(  X,0\right)  dX-\int_{0}^{t}\left[
\int\partial_{s}\psi\left(  X,s\right)  H\left(  X,s\right)  dX\right]
ds\\
&  = -\int_{0}^{t}\left[  \int\psi\left(  X,s\right)  \mathcal{A}\left[
H\right]  \left(  X,s\right)  H\left(  X,s\right)  dX\right]  ds+\int
_{0}^{t}\left[  \int\psi\left(  X,s\right)  \mathcal{Q}\left[  H\right]
\left(  X,s\right)  dX\right]  ds   \,.
\end{split}
\end{equation}
\end{definition}

We have the following result.

\begin{lemma}
\label{L1}Suppose that $H\in C\left(  \left[  0,T\right]  ;\mathcal{X}_{\rho
}\right)  $ is a mild solution of (\ref{D4})-(\ref{D5}).  Then, it is also a weak solution of (\ref{D4})-(\ref{D5})
in the sense of Definition \ref{def2}.
\end{lemma}

\begin{proof}
We have seen in the proof of Lemma \ref{L.WPR} that  $\mathcal{A}\left[  H\right]  $ is continuous in $X$ and $t$  and bounded 
and that $\mathcal{Q}\left[
H\right]  $ is a locally bounded measure. 
Hence we can take the time derivative in the weak formulation of \eqref{D8a}, that is after multiplying \eqref{D8a} with 
$\psi \in C^0_0([0,\infty))$ and integrating.
We can do the same if $\psi=\psi(X,s)$ with $\psi \in C^1_0([0,\infty) \times [0, t])$ which implies the statement of the Lemma after integrating
over $s$.
\end{proof}

We can now use the weak formulation for $H$ to show that we can extend the local solution for all times $t>0$.

\begin{proposition}
\label{WPR}
 For any $\lambda>0$ and any $h_{0}\in\mathcal{X}_{\rho}$
there exists a unique mild solution of (\ref{D4})-(\ref{D5}) for all times $t>0$.
Moreover, for any $T>0$ there exists a constant $C(T)$ that is independent of $\lambda$
such that
\begin{equation}\label{S1E3a}
\sup_{0\leq t \leq T} \left\Vert H\left(  \cdot,t\right)  \right\Vert _{\mathcal{X}_{\rho}}\leq
C\left(  T\right) \,.
\end{equation}
\end{proposition}
\begin{proof}
The local  solution of Lemma \ref{L.WPR}  can be extended in time as long as we have a uniform estimate for
$\left\Vert H\left(  \cdot,t\right)  \right\Vert _{\mathcal{X}_{\rho}}.$

In order to derive this estimate, we recall the well-known reformulation of the nonlinear coagulation term, stated here for general
functions $\psi, h$ and $K$.

\begin{equation}\label{rearrange}
\begin{split}
\int\psi(x)  &\int_{0}^{\infty}\frac{K\left(
x,z\right)  }{z}h\left(  z\right)  h\left(  x\right) \,dz\,dx 
  -\int\psi\left(  x\right)  \int_{0}^{x}\frac{K\left(
y,x-y\right)  }{\left(  x-y\right)  }h\left(  x-y\right)  h\left(
y\right) \,dy\,dx\\
 =&\int\psi(x)  \int_{0}^{\infty}\frac{K\left(
x,z\right)  }{z}h\left(  z\right)  h(x)\,dz\,dx  
  -\int \int_{y}^{\infty}\frac{K\left(  y,x-y\right)  }{\left(
x-y\right)  }h\left(  x-y\right)  h\left(  y\right)  \psi\left(
x\right) \,dx dy\\
  =&\int\psi(x)\int_{0}^{\infty}\frac{K\left(
x,z\right)  }{z}h\left(  z\right)  h(x)\,dz\,dx  
  -\int \int_{0}^{\infty}\frac{K\left(  y,z\right)  }{z}h\left(
z\right)  h\left(  y\right)  \psi\left(  z+y\right)  dz\,dy\\
 =&\int dx\int_{0}^{\infty}\frac{K\left(  x,z\right)  }{z}h\left(
z\right)  h\left(  x\right)  \left[  \psi\left(  x\right)  -\psi\left(
z+x\right)  \right]  dz\,.
\end{split}
\end{equation}
We use now  \eqref{rearrange} in \eqref{S2} for a nonnegative test function that is independent of $t$ and decreasing (the
different arguments $X e^{-\beta t}$ etc. in $K_{\lambda}$  do not affect \eqref{rearrange}).
This implies due to the positivity of $H$ that 
\begin{equation}\label{Hest}
  \int H(X,t) \psi(X)\,dX \leq \int h_0(X) \psi(X)\,dX + \beta \rho \int_0^t \int H(X,s) \psi(X)\,dX\,ds\,.
\end{equation}
We can now consider a sequence of test functions $\psi$ that approach the characteristic function on $[0,R]$. 
By a Gronwall argument we obtain  
(\ref{S1E3a}) with $C\left(  T\right)  $ independent of $\lambda.$
This implies that the solution $H$ is globally defined in time. 
\end{proof}

We can define weak solutions of (\ref{D2})-(\ref{D3}) in the same spirit as
in Definition \ref{def2}.

\begin{definition}
\label{def3} We will say that $h$ is a weak solution of (\ref{D2})-(\ref{D3})
in $\left[  0,\infty\right)  \times\left[  0,T\right]  $ if for any $
{t}\in\left[  0,T\right]  $ and any test function $\psi\in C_{0}^{1}\left(
\left[  0,\infty\right)  \times\left[  0,{t}\right]  \right)  $ we have
\begin{equation}\label{S4}
\begin{split}
&  \int h\left(  x,{t}\right)  \psi\left(  x,{t}\right)  dx-\int
h_{0}\left(  x\right)  \psi\left(  x,0\right)  dx-\int_{0}^{t}\left[
\int\partial_{s}\psi\left(  x,s\right)  h\left(  x,s\right)  dx\right]
ds \\
&  +\int_{0}^{t}\left[  \int\psi(x,s)  \int_{0}%
^{\infty}\frac{K_{\lambda}(x,z)}{z}h(z,s)
h(x,s)\,dz\,dx  \right] \,ds \\
&  -\int_{0}^{{t}}\left[  \int\psi(x,s)\int_{0}%
^{x}\frac{K_{\lambda}\left(  y,x-y\right)  }{\left(  x-y\right)  }h\left(
x-y,s\right)  h\left(  y,s\right) \,dy\,dx \right] \,ds \\
&  +\beta\int_{0}^{{t}}\int\partial_{x}\left(  x\psi\right)  h\left(
x,s\right)  dx\,ds-\beta\rho\int_{0}^{{t}}\int\psi\left(  x,s\right)
h\left(  x,s\right)  dx\,ds =0 \,.
\end{split}
\end{equation}

\end{definition}

By a simple change of variables and adapting the test functions correspondingly we obtain the following Lemma.
\begin{lemma}
\label{L2}Suppose that $H\in C\left(  \left[  0,T\right]  ;\mathcal{X}_{\rho
}\right)  $ is a solution of (\ref{D4})-(\ref{D5}) in the sense of Definition
\ref{def1}. Then, the function $h$ defined by means of (\ref{D6}) is also a
weak solution of (\ref{D2})-(\ref{D3}) in the sense of Definition \ref{def3}.
\end{lemma}

It will be convenient in the following to work with the notion of weak solutions, in particular
 for the proof of the weak continuity of the
evolution semigroup.
Given $h_{0}\in\mathcal{X}_{\rho}$ and $H$ as in Proposition
\ref{WPR} we can define $h$ as in (\ref{D6}). We will write, for
any $\lambda>0$%
\begin{equation}
S_{\lambda}\left(  t\right)  h_{0}=h\left(  \cdot,t\right)  \ ,\ \ t\geq0\,.
\label{S1}%
\end{equation}

Proposition \ref{WPR} implies that
 $S_{\lambda}\left(  t\right)  $ maps $\mathcal{X}_{\rho}$ into
itself. Moreover, it satisfies the usual properties satisfied by evolution
semigroups
\[
S_{\lambda}\left(  t_{1}\right)  S_{\lambda}\left(  t_{2}\right)  =S_{\lambda
}\left(  t_{1}+t_{2}\right)   \ ,\ \ t_{1},t_{2}\in\left[  0,\infty\right)
\  ,\ \ \ S_{\lambda}\left(  0\right)  =I\,.
\]

Our next goal is to show that the maps $S_{\lambda}\left(  t\right)  $ are
continuous in the weak topology.

We introduce, for further reference, the following auxiliary semigroup. Given
$h_{0}\in\mathcal{X}_{\rho}$ and $H_{\lambda}$ as in Proposition \ref{WPR} we
define
\begin{equation}
T_{\lambda}\left(  t\right)  h_{0}=H_{\lambda}\left(  \cdot,t\right)
\ \ ,\ \ t\geq0 \label{S5}%
\end{equation}

We remark that $T_{\lambda}\left(  t\right)  $ also satisfies  the semigroup
properties.

\subsection{Continuity of $S_{\lambda}\left(  t\right)  $ in the weak
topology}

\begin{proposition}
\label{weakcontin} A closed ball in   $\mathcal{X}_{\rho}\subset\mathcal{M}^{+}\left(
\left[  0,\infty\right)  \right)  $ is a compact subset of
 $\mathcal{M}^{+}\left(  \left[  0,\infty\right)  \right)  $ 
 endowed with the weak topology. The transformation $S_{\lambda
}\left(  t\right)  $ defined by means of (\ref{S1}) for any $t\in\left[
0,T\right]  $ is a continuous map from $\mathcal{X}_{\rho}$ into itself.
\end{proposition}

\begin{corollary}
The mapping $S_{\lambda}\left(  t\right)  :\mathcal{X}_{\rho}\rightarrow
\mathcal{X}_{\rho}$ is compact for any $t\in\left[  0,T\right]  $ if
$\mathcal{X}_{\rho}$ is endowed with the weak topology.
\end{corollary}

\begin{remark}
The continuity that we obtain is not uniform in $\lambda.$
\end{remark}

\begin{proof}
[Proof of Proposition \ref{weakcontin}]The transformation (\ref{D6}) is
continuous in the weak topology by the fact that the
adjoint transformation, that acts on the test functions, brings the space
$C^{0}\left[  0,\infty\right)  $ to itself and  is continuous in the $L^{\infty}$-norm.
 Therefore, we just need to check that the transformations $T_{\lambda
}\left(  t\right)  $ given in (\ref{S5}) is continuous in the weak topology.
Since the transformation is nonlinear it is not sufficient to check continuity
at $h_{0}=0.$ More precisely, let us fix some time ${t}\in\left[
0,T\right]  $ and consider  a test function $\bar{\psi}\left(
X\right)  ,\ \bar{\psi}\in C_{0}^{1}(\left[  0,\infty\right)).$ Suppose that
we have two functions $H_{1},\ H_{2}$ such that $T_{\lambda}\left(  t\right)
h_{0,k}=H_{k}\left(  \cdot,t\right)  ,$ $k=1,2.$ We want to show that
$\int\bar{\psi}\left[  H_{1}-H_{2}\right]  dX$ can be made small if
$h_{0,1},\ h_{0,2}$ are close in the sense of the weak topology.

To this end we will construct a  suitable function $\psi\left(  X,t\right)$ such 
that $\psi\left(  X,t\right)  =\bar{\psi
}\left(  X\right) $ and
\begin{equation}\label{psidef}
\int\left(  H_{1}\left(  X,{t}\right)  -H_{2}\left(  X,{t}\right)
\right)  \bar{\psi}\left(  X\right)  dX=\int\left(  h_{0,1}\left(  X\right)
-h_{0,2}\left(  X\right)  \right)  \psi\left(  X,0\right)  dX
\end{equation}
To see how to define $\psi$, observe that due to Lemma \ref{L1} we have that $H_{1}$ and $H_{2}$
satisfy (\ref{S2}). Subtracting these equations we obtain
\begin{equation}\label{T1E3}
\begin{split}
  \int&\left(  H_{1}\left(  X,{t}\right)  -H_{2}\left(  X,{t}\right)
\right)  \psi\left(  X,{t}\right)  dX 
-\int\left(  h_{0,1}\left(  X\right)
-h_{0,2}\left(  X\right)  \right)  \psi\left(  X,0\right)  dX\\
&-\int_{0}%
^{{t}}\left[  \int\partial_{s}\psi\left(  X,s\right)  \left(  H_{1}\left(
X,s\right)  -H_{2}\left(  X,s\right)  \right)  dX\right]  ds\\
&  +\int_{0}^{{t}}\left[  \int\psi\left(  X,s\right)  \left(
\mathcal{A}\left[  H_{1}\right]  \left(  X,s\right)  H_{1}\left(  X,s\right)
-\mathcal{A}\left[  H_{2}\right]  \left(  X,s\right)  H_{2}\left(  X,s\right)
\right)  dX\right]  ds\\
&-\int_{0}^{{t}}\left[  \int\psi\left(  X,s\right)
\left(  \mathcal{Q}\left[  H_{1}\right]  \left(  X,s\right)  -\mathcal{Q}%
\left[  H_{2}\right]  \left(  X,s\right)  \right)  dX\right]  ds
  =0 \,.
\end{split}
\end{equation}

We need to transform the last two integral terms on the left-hand side of
(\ref{T1E3}). To this end note that
\[
\begin{split}
  \int &\psi\left(  X,s\right)  \left(  \mathcal{A}\left[  H_{1}\right]
\left(  X,s\right)  H_{1}\left(  X,s\right)  -\mathcal{A}\left[  H_{2}\right]
\left(  X,s\right)  H_{2}\left(  X,s\right)  \right)  dX\\
 = &\int\psi\left(  X,s\right)  \Big(  \int_{0}^{\infty}\frac{K_{\lambda
}\left(  Xe^{-\beta s},Ze^{-\beta s}\right)  }{Z}\big ( H_{1}(X,s) H_{1}(Z,s)
-H_2(X,s) H_2(Z,s)\big)  dZ\,dX 
\\
&\qquad
  -\beta\rho\int\psi\left(  X,s\right)  \left(  H_{1}\left(  X,s\right)
-H_{2}\left(  X,s\right)  \right)  dX\\
= &\frac{1}{2}\int\psi(X,s) \Big(  \int_{0}^{\infty}%
\frac{K_{\lambda}\left(  Xe^{-\beta s},Ze^{-\beta s}\right)  }{Z}
\Big[ H_{1}(X,s)  \left(  H_{1}(Z,s)  -H_{2}(Z,s)  \right) \\
&\qquad \qquad \qquad +\left(  H_{1}(X,s)  -H_{2}(X,s)  \right)  H_{2}( Z,s)  \Big]  dZ\Big)  \,dX\\
&  \qquad +\frac{1}{2}\int\psi(X,s) \Big(  \int_{0}^{\infty}%
\frac{K_{\lambda}( Xe^{-\beta s},Ze^{-\beta s})  }{Z} \Big[(  H_{1}(X,s)-H_{2}(X,s))
H_{1}(Z,s) \\
& +H_{2}(X,s)(  H_{1}(Z,s)  -H_{2}(Z,s))  \Big]  dZ\Big)\,dX   -\beta\rho\int\psi(X,s) (  H_{1}(X,s)-H_{2}(X,s))  dX\\
 = &\int\psi(X,s)  \int_{0}^{\infty}\frac{K_{\lambda}\left(
Xe^{-\beta s},Ze^{-\beta s}\right)  }{Z}  \cdot \Big[ \frac{H_{1}(X,s)  +H_{2}(X,s)  }%
{2}(  H_{1}(Z,s)  -H_{2}(Z,s))\\
&\qquad \qquad \qquad 
+\frac{H_{1}(Z,s)  +H_{2}(Z,s)}{2}(
H_{1}(X,s)  -H_{2}(X,s) )\Big]  dZ\,dX\\
& \qquad  -\beta\rho\int\psi(X,s) (  H_{1}(X,s)-H_{2}(X,s))  dX\,.
\end{split}
\]
Hence
\[
\begin{split}
&  \int\psi\left(  X,s\right)  \left(  \mathcal{A}\left[  H_{1}\right]
\left(  X,s\right)  H_{1}\left(  X,s\right)  -\mathcal{A}\left[  H_{2}\right]
\left(  X,s\right)  H_{2}\left(  X,s\right)  \right)  dX\\
&  =\frac{1}{2}\int\psi(X,s) \int_{0}^{\infty}\Big(
\frac{G\left(  Z,X,s\right)  }{Z}\left(  H_{1}\left(  Z,t\right)
-H_{2}\left(  Z,s\right)  \right) \\
& +\frac{G\left(  X,Z,s\right)  }{Z}\left(
H_{1}\left(  X,s\right)  -H_{2}\left(  X,s\right)  \right)  \Big)  dZ\,dX
  -\beta\rho\int\psi\left(  X,s\right)  \left(  H_{1}\left(  X,s\right)
-H_{2}\left(  X,s\right)  \right)  dX\,,
\end{split}
\]
where
\[
G\left(  X,Z,s\right)  =K_{\lambda}\left(  Xe^{-\beta s},Ze^{-\beta s}\right)
\frac{H_{1}\left(  Z,s\right)  +H_{2}\left(  Z,s\right)  }{2}\,.
\]

Exchanging the labels $Z$ and $X$ in the first term on the right hand side,  we find
\begin{align*}
 \int\psi&\left(  X,s\right)  \left(  \mathcal{A}\left[  H_{1}\right]
\left(  X,s\right)  H_{1}\left(  X,s\right)  -\mathcal{A}\left[  H_{2}\right]
\left(  X,s\right)  H_{2}\left(  X,s\right)  \right)  dX\\
 =&\int \int_{0}^{\infty}\Big(  \frac{G\left(  X,Z,s\right)  }{X}\left(
H_{1}\left(  X,s\right)  -H_{2}\left(  X,s\right)  \right)  \psi\left(
Z,s\right)  \\
&+\frac{G\left(  X,Z,s\right)  }{Z}\left(  H_{1}\left(  X,s\right)
-H_{2}\left(  X,s\right)  \right)  \psi\left(  X,s\right)  \Big)  dZ\,dX\\
&  -\beta\rho\int\psi\left(  X,s\right)  \left(  H_{1}\left(  X,s\right)
-H_{2}\left(  X,s\right)  \right)  dX\\
  =&\int \left(  H_{1}\left(  X,s\right)  -H_{2}\left(  X,s\right)  \right)
\int_{0}^{\infty}\left(  \frac{G\left(  X,Z,s\right)  }{X}\psi\left(
Z,s\right)  +\frac{G\left(  X,Z,s\right)  }{Z}\psi\left(  X,s\right)  \right)
dZ\,dX\\
&  -\beta\rho\int\psi\left(  X,s\right)  \left(  H_{1}\left(  X,s\right)
-H_{2}\left(  X,s\right)  \right)  dX\,.
\end{align*}
In exactly the same way we obtain
\begin{align*}
&  \int\psi\left(  X,s\right)  \left(  \mathcal{Q}\left[  H_{1}\right]
\left(  X,s\right)  -\mathcal{Q}\left[  H_{2}\right]  \left(  X,s\right)
\right)  dX\\
&  =\int \left(  H_{1}\left(  X,s\right)  -H_{2}\left(  X,s\right)  \right)
\left(  \int_{0}^{\infty}\psi\left(  X+Z,s\right)  \left(  \frac{G\left(
X,Z,s\right)  }{Z}+\frac{G\left(  X,Z,s\right)  }{X}\right)  dZ\right)\,dX\,.
\end{align*}
In summary,  equation (\ref{T1E3}) can be rewritten as
\[
\begin{split}
   \int&\left( H_{1}(X,{t})  -H_{2}(X,{t}) \right)  \psi(X,{t})  dX-\int\left(  h_{0,1}(  X)
-h_{0,2}( X)  \right)  \psi(X,0)  dX\\
=& \int_0^{{t}}\left[  \int\partial_{s}\psi(X,s) (H_{1}(X,s)  -H_{2}(X,s))  dX\right]  ds\\
&  -\int_{0}^{{t}}\left[  \int  (  H_{1}(X,s) -H_{2}(X,s) )  
\int_{0}^{\infty}\left(  \frac{G(X,Z,s)}{X}\psi(Z,s)  +\frac{G(X,Z,s)}{Z}\psi(X,s)  \right)  dZ\,dX\right]  ds\\
&  +\int_{0}^{{t}}\left[  \int (  H_{1}(X,s)-H_{2}(X,s))  \int_{0}^{\infty}\psi(X+Z,s)
\left(  \frac{G(X,Z,s)}{Z}+\frac{G(X,Z,s)  }{X}\right)  dZ\,dX\right]  ds\\
&-\int_0^t  \beta\rho\int\psi\left(  X,s\right)  \left(  H_{1}\left(  X,s\right)
-H_{2}\left(  X,s\right)  \right)  dX\,dt\,.
\end{split}
\]
Hence, in order to obtain \eqref{psidef}, we choose $\psi$ as the solution of the following equation
\begin{align}
  \partial_{s}\psi\left(  X,s\right)  &=\int_{0}^{\infty}\left(
\frac{G\left(  X,Z,s\right)  }{X}\psi\left(  Z,s\right)  +\frac{G\left(
X,Z,s\right)  }{Z}\psi\left(  X,s\right)  \right)  dZ\nonumber\\
& \quad  -\int_{0}^{\infty}\psi\left(  X+Z,s\right)  \left(  \frac{G\left(
X,Z,s\right)  }{Z}+\frac{G\left(  X,Z,s\right)  }{X}\right)  dZ - \beta\rho\psi\left(  X,s\right)   \,. \label{T1E4}%
\end{align}
with initial value
\begin{equation}
\psi\left(  X,{t}\right)  =\bar{\psi}\left(  X\right)\,.  \label{T1E5}%
\end{equation}
This equation can be solved for any $\lambda>0$ in the class of functions
$\psi\in C^{1}\left(  \left[  0,{t}\right]  :C\left[  0,\infty\right)
\right)$ such that $\sup_{X \geq 0} (1+X) |\psi(X)| < \infty$
(see Lemma \ref{L.adjoint} below).
 and thus \eqref{psidef} holds. 
Due to the decay of $\psi(X,0)$, the fact that $h_{0,1}-h_{0,2} \in \mathcal{X}_{\rho}$ and \eqref{dyadic}
the function $\psi\left(  X,0\right)  $ can be
replaced by a function with compact support and this finishes the proof.
\end{proof}

In the proof of Proposition \ref{weakcontin} we have used the well-posedness
of the problem (\ref{T1E4})-(\ref{T1E5}).

\begin{lemma}
\label{L.adjoint}
Let us define the Banach space $\mathcal{Z}$ as the space of functions
$\varphi\in C\left[  0,\infty\right)  $ satisfying $\left\Vert \varphi
\right\Vert _{\mathcal{Z}}=\sup_{X\geq0}\left(  1+X\right) 
 |\varphi\left(  X\right)|  <\infty.$ For any $\bar \psi\in C_{0}^{1}\left[
0,\infty\right)  $ there exists a unique solution $\psi$ of (\ref{T1E4})-(\ref{T1E5})
 such that $\psi\in C^{1}\left(  \left[  0,{t}\right]
:\mathcal{Z}\right)  .$
\end{lemma}

\begin{proof}
Note that due to the growth of $K_{\lambda}$, the function $G$ can in an averaged sense be estimated by $X^{\gamma} Z^{-\rho} +
Z^{\gamma-\rho}$, while $H$ can be estimated in an average sense by $Z^{-\rho}$. Furthermore, recall that $K_{\lambda}$ vanishes if $X,Z \leq \lambda/2$ and
$Z \leq \frac{\lambda}{2-\lambda}X$ etc.. 
Therefore, we can prove as in previous arguments, using \eqref{dyadic}, 
that all the integral operators in \eqref{T1E4} map  $\mathcal{Z}$ into itself. Since \eqref{T1E4} is a linear equation for $\psi$, the statement of the lemma
follows from a standard fixed point argument. We omit the details here. 
\end{proof}

\section{Invariance of the set $\mathcal{Y}$ under the semigroups $S_{\lambda
}\left(  t\right)  .$}
\label{S.invariance}

The next goal is to show that the set $\mathcal{Y}$ defined by means of
(\ref{B1a})-(\ref{B2a}) remains invariant under the action of the semigroup
$S_{\lambda}\left(  t\right)  .$ We first prove the invariance of the upper
estimate (\ref{B1a}).

\subsection{Invariance of the set defined by means of (\ref{B1a}).}

\begin{proposition}\label{P.upperbound}
Suppose that $h_{0}$ satisfies (\ref{B1a}) and that $h\left(  \cdot,t\right)  $ is given 
by (\ref{S1}). Then $h\left(  \cdot,t\right)  $ satisfies (\ref{B1a}) as well.
\end{proposition}

\begin{proof}
This is a corollary of Proposition \ref{WPR}.
Indeed, taking in \eqref{Hest} as test function (after some approximation procedure), $\psi(X)=\chi_{[0,R]}(X)$, we obtain by 
Gronwall's inequality and the fact that $h_0$ satisfies \eqref{B1a}
\[
 \int_0^R H(X,t)\,dX \leq R^{1{-}\rho} e^{\beta \rho t}\,.
\]
Now a change of variables from $X$ to $x$ gives

\[
\int_0^{Re^{-\beta{t}}}h\left(  x,{t}\right)
dx\leq\left(  Re^{-\beta{t}}\right)  ^{1{-}\rho}\  ,\ \ {t}\geq0\,.
\]
Since  $R$ is arbitrary, it  follows that
\[
\int_0^{R}h\left(  x,{t}\right)  dx\leq  R^{1{-}\rho}\,,\ \qquad \mbox{ for all } R\geq 0 \mbox{ and } {t}\geq0
\]
and the result follows.
\end{proof}

\subsection{The dual problem}\label{S.massflux}

Our next goal is to show that also property (\ref{B2a}) is invariant under the semigroup $S_{\lambda}\left(
t\right)  .$ In order to prove this we derive first a formula that allows us to
compute the change of mass fluxes. More precisely, we compute the rate of
change of $\int\psi\left(  x,t\right)  h\left(  x,t\right)  dx$ for some
particular class of test functions.

Recalling \eqref{rearrange} we find 
that $h$ satisfies 
\begin{equation}\label{T2E1}
 \begin{split}
  \partial_{t}&\left(  \int h\left(  x,t\right)  \psi\left(  x,t\right)
dx\right)   =\int_{0}^{\infty}h\left(  x,t\right)  dx\int_{0}^{\infty}\frac{K\left(
x,z\right)  h\left(  z,t\right)  }{z}\left[  \psi\left(  x+z,t\right)
-\psi\left(  x,t\right)  \right]  dz\\
& \quad -\beta\int x\partial_{x}\psi\left(  x,t\right)  h\left(  x,t\right)
dx+\beta\int\left(  \rho{-}1\right)  h\left(  x,t\right)  \psi\left(
x,t\right)  dx+\int h\left(  x,t\right)  \partial_{t}\psi\left(  x,t\right)
dx \,.
\end{split}\end{equation}

Using Fubini's Theorem we also have the analogous result for weak solutions. 

\begin{lemma}
Suppose that $h$ is a weak solution of (\ref{D2})-(\ref{D3}) in $\left[
0,\infty\right)  \times\left[  0,T\right]  ,$ in the sense of Definition
\ref{def3}. Then for any $\psi\in C_{0}^{1}\left(  \left[  0,\infty\right)
\times\left[  0,T\right]  \right)  $ we have
\begin{equation}\label{S6}
 \begin{split}
  \int h\left(  x,{t}\right) & \psi\left(  x,{t}\right)  dx-\int
h_{0}\left(  x\right)  \psi\left(  x,0\right)  dx-\int_{0}^{{t}}\left[
\int\partial_{s}\psi\left(  x,s\right)  h\left(  x,s\right)  dx\right]
ds\\
&  -\int_{0}^{{t}}\left[  \int h\left(  x,s\right)  \int_{0}^{\infty
}\frac{K_{\lambda}\left(  x,z\right)  }{z}h\left(  z,s\right)  \left[
\psi\left(  x+z,s\right)  -\psi\left(  x,s\right)  \right] \,dz\,dx \right]\,ds\\
&  +\beta\int_{0}^{{t}}\int x\partial_{x}\psi\left(  x,s\right)
h\left(  x,s\right) \, dx\,ds-\beta\left(  \rho{-}1\right)  \int_{0}^{{t}}%
\int\psi\left(  x,s\right)  h\left(  x,s\right)  \,dx\,ds  =0\,.
\end{split}
\end{equation}
\end{lemma}

Formula (\ref{S6}) suggests a particularly convenient choice of the test
functions $\psi$  as described in the introduction.  More precisely, suppose that we choose $\psi\left(
x,s\right)  $ as the solution of the equation
\begin{equation}
-\partial_{s}\psi\left(  x,s\right)  -\int_{0}^{\infty}\frac{K_{\lambda
}\left(  x,z\right)  }{z}h\left(  z,s\right)  \left[  \psi\left(
x{+}z,s\right)  -\psi\left(  x,s\right)  \right]\,dz  +\beta x\partial_{x}%
\psi\left(  x,s\right)  -\beta\left(  \rho{-}1\right)  \psi\left(  x,s\right)
=0. \label{S7}%
\end{equation}

Then (\ref{S6}) becomes
\begin{equation}
\int h\left(  x,{t}\right)  \psi\left(  x,{t}\right)  dx=\int
h_{0}\left(  x\right)  \psi\left(  x,0\right)  dx\,. \label{S8}%
\end{equation}

In order to check (\ref{B2a}) we need to estimate quantities like $\int
_{0}^{R}h\left(  x,{t}\right)  dx.$ This suggests to study the solutions
of (\ref{S7}) such that $\psi\left(  x,{t}\right)  $ is the characteristic
function of the interval $\left[  0,R\right]  .$ We will next prove that the
solution of such a problem exists and construct a subsolution.

\subsection{Solvability of the dual problem }

We can simplify (\ref{S7}) by a change of variables
\begin{equation}
\psi\left(  x,s\right)  =\exp\left(  -\beta\left(  \rho{-}1\right)  \left(
s-{t}\right)  \right)  \Psi\left(  X,s\right)  \,,\ \ X=xe^{\beta\left(
s-{t}\right)  } \label{T2E2}%
\end{equation}
that transforms (\ref{S7}) into
\begin{equation}
\partial_{s}\Psi\left(  X,s\right)  +\int_{0}^{\infty}\frac{K_{\lambda
}\left(  Xe^{-\beta\left(  s-{t}\right)  },Ze^{-\beta\left(  s-{t}\right) 
 }\right)  }{Z}h\left(  Ze^{-\beta\left(  s-{t}\right)
},s\right)  \left[  \Psi\left(  X{+}Z,s\right)  {-}\Psi\left(  X,s\right)
\right]\,dZ  =0 \label{T2E3}%
\end{equation}
with initial data
\begin{equation}
\Psi\left(  X,{t}\right)  =\chi_{\left[  0,R\right]  }\left(  X\right)\,.
\label{T2E4}%
\end{equation}

\begin{lemma}
 There exists a unique  solution $\Psi \in L^{\infty}(0,\infty) \times [0, t])$
 to \eqref{T2E3}-\eqref{T2E4}. It satisfies $\Psi(X,s)=0$ for all $X>R$ and $s \in [0, t]$.
\end{lemma}
\begin{proof}
Recall that $K_{\lambda}$ vanishes for small $Z$. Hence the well-posedness of the equation in the class of bounded functions
follows by a standard fixed point argument.
\end{proof}

To derive more quantitative information about the solutions of
(\ref{T2E3})-(\ref{T2E4}) we
construct a suitable subsolution for (\ref{T2E3})-(\ref{T2E4}).

For the following we write $\tau =t-s$,  $\tilde \Psi(X,\tau)=\Psi(X,s)$ and 
\[
Q\left(  X,Z,\tau\right)  =\frac{K_{\lambda}\left(  Xe^{-\beta\left(  s-
{t}\right)  },Ze^{-\beta\left(  s-{t}\right)  }\right)  }{Z}h\left(
Ze^{-\beta\left(  s-{t}\right)  },s\right)\,,
\]
such that
\begin{equation}\label{T2E7}
 \partial_\tau \tilde \Psi(X,\tau) - \int_0^{\infty} Q(X,Z,\tau) \big[\tilde \Psi(X+Z,\tau) - \tilde \Psi(X,\tau)\big]\,dZ =0\,, \qquad 
\tilde \Psi(X,0)=\chi_{[0,R]}(X)\,.
\end{equation}

Due to the properties of $h$ and \eqref{dyadic} we have for $0\leq \tau\leq1$ and $X \leq R$ that
\begin{equation}
\int_{R}^{\infty}Q\left(  X,Z,\tau\right)  dZ\leq K  R^{\gamma{-}\rho}\label{T2E6}%
\end{equation}
and this  motivates to consider the auxiliary problem 
\begin{equation}
\partial_{\tau}\hat{\psi}\left(  X,\tau\right)  -\int_{0}^{\infty
}\frac{1}{Z^{1+\rho-\gamma}}\left[ \hat{ \psi}\left(  X+Z,\tau\right)
-\hat{\psi}\left(  X,\tau\right)  \right]\,dZ  =0\,,\ \ \ \hat{\psi}\left(  X,0\right)
=\chi_{[0,R]  }\left(  X\right) \,. \label{T2E8}%
\end{equation}

It is natural to look for self-similar solutions of (\ref{T2E8}). Note that
$\psi\left(  X,\tau\right)  =0$ if $X\geq R.$ We change variables as 
\[
a=\rho-\gamma>0\,, \qquad Y = \frac{R-X}{\tau^{\frac{1}{a}}}, \qquad Z=\tau^{\frac{1}{a}} \eta
\]
and look for solutions of the form $ \hat \psi(X,\tau)=W(Y)$.
Then $W$ solves 
\begin{equation}
\frac{1}{a}Y W'(Y)=\int_{0}^{\infty}\frac{1}{\eta^{1+a}}\left[
W\left(  Y\right)  -W\left(  Y-\eta\right)  \right]\,d\eta \,, \qquad Y \in (0,\infty)\,, \label{T3E4}%
\end{equation}
and
\begin{equation}\label{T3E4b}
W\left(  Y\right)  =0\qquad \mbox{ for }  Y\leq 0\,,\qquad  W\left(  Y\right)
\rightarrow 1\text{ as }Y\rightarrow \infty\,.
\end{equation}

\begin{proposition}\label{P.W}
 There exists a unique positive solution to \eqref{T3E4}-\eqref{T3E4b}. It is increasing and satisfies 
\begin{equation}\label{Wdecay}
W'(Y) \sim \frac{C}{Y^{1+a} } \qquad \mbox{ as } Y \to \infty\,.
\end{equation}
\end{proposition}
\begin{proof}
We can solve \eqref{T3E4}-\eqref{T3E4b} explicitly via Laplace transform. Indeed, the Laplace transform  
$\hat W(p)= \int_0^\infty e^{-Yp} W(Y)\,dY$ of $W$ solves
\[
 - \frac{1}{a} \big( \hat W(p) + p {\hat W}'(p)\big) = \hat W(p) p^a \int_0^{\infty} \frac{1-e^{-z}}{z^{1+a}}\,dz =
\hat W(p) p^a \frac{\Gamma(1{-}a)}{a}\,,
\]
that is
\begin{equation}\label{Wlaplace}
 {\hat W}'(p) = - \hat W(p) \frac{1}{p} \Big( 1 + p^a \Gamma(1{-}a)\Big)\,.
\end{equation}
Together with the constraint that $1-W(Y) \to 0 $ as $Y \to \infty$, the solution is uniquely determined and given by 
$\hat W(p) = \frac{1}{p} e^{-\frac{\Gamma(1{-}a)}{a} p^a}$. Now $W(Y)$ can be computed using the inverse Laplace transform together
with contour integration. We obtain
\[
 W(Y) =1+ \frac{1}{\pi} \int_0^{\infty} \frac{1}{p} e^{-Yp -\frac{\Gamma(1{-}a)}{a} p^a \cos(\pi a)} 
\sin\Big( \frac{\Gamma(1{-}a)}{a} \sin(\pi a) p^a\Big) \,dp
\]
and
\[
 W'(Y) = -\frac{1}{\pi Y} \int_0^{\infty} e^{-z - \frac{\Gamma(1{-}a)}{a} (\frac{z}{Y})^a \cos(\pi a)}
\sin \Big ( \frac{\Gamma(1{-}a)}{a}\sin(\pi a) \frac{z^a}{Y^a}\Big) \,dz\,.
\]
From this the decay behaviour \eqref{Wdecay} follows immediately.
\end{proof}

\subsection{Comparison argument }

\begin{lemma}\label{L.comparison}
 Let $\tilde \Psi(X,\tau)$ be the solution of \eqref{T2E7} and $\hat \psi(X,\tau)$ be the solution of 
\begin{equation}
\partial_{\tau}\hat \psi\left(  X,\tau\right)  -M\int_{0}^{\infty
}\frac{dZ}{Z^{1+\rho-\gamma}}\left[  \hat\psi\left(  X+Z,\tau\right)
-\hat\psi\left(  X,\tau\right)  \right]  =0\, ,\ \ \ \hat\psi\left(  X,0\right)
=\chi_{[0,R] }\left(  X\right)  \ \label{T3E2}%
\end{equation}
with $M>0.$ Then $\tilde \Psi(X,\tau)\geq\hat \psi(X,\tau)$ for all $X,\tau\geq 0$ if $M$ is sufficiently large.
\end{lemma}
\begin{proof}
Since the constant $M$ can be absorbed into the time scale, it
follows that the solution of problem \eqref{T3E2}  can be written as
\[
\hat \psi\left(  X,\tau\right)  =W(Y) 
\,,\qquad \mbox{ with } \quad Y = \frac{R-X}{\big (M\tau\big)^{\frac{1}{a}}}\,,
\]
where $W$ solves (\ref{T3E4}). Since $\hat \psi(X,0)=\Psi(X,0)=\chi_{[0,R]}(X)$ it remains to show that 
\begin{equation}
\partial_{\tau}\hat \psi\left(  X,\tau\right)  \leq\int_{0}^{\infty} Q\left(
X,Z,\tau\right)  \left[  \hat\psi\left(  X+Z,\tau\right)  -\hat\psi\left(  X,\tau\right)
\right]  \,dZ \label{T3E5}%
\end{equation}
for a sufficiently large  $M$.
Using $\hat \psi\left(  X,\tau\right)
 =W\left(  Y\right)$
we find that (\ref{T3E5}) reduces to
\[
 \frac{1}{a}\frac{YW^{\prime}\left(  Y\right)  }{\tau}\geq\int_{0}^{\infty
}Q\left(  X,Z,\tau\right)  \left[   W\left(  \frac{R-X }{\left(  M\tau\right)  ^{\frac
{1}{a}}}\right) -  W\left(  \frac{R-(X+Z)
}{\left(  M\tau\right)  ^{\frac{1}{a}}
}\right)\right]\,dZ \,.
\]

By (\ref{T3E4}) we obtain that this inequality is equivalent to
\[
\begin{split}
\frac{1}{\tau}\int_{0}^{\infty}\frac{1}{\eta^{1+a}}& \left[  W(Y)  {-}W(Y{-}\eta)  \right]\,d\eta\\
&  \geq\int_{0}^{\infty} Q\left(
X,Z,\tau\right)  \left[  W\left(  \frac{\left(  R{-}X\right)  }{\left(
M\tau\right)  ^{\frac{1}{a}}}\right)  -W\left(
\frac{\left(  R{-}(X{+}Z)\right)  }{\left(  M\tau\right)  ^{\frac{1}{a}}}\right)  \right]\,dZ\\
&\geq\left(  M\tau\right)  ^{\frac
{1}{a}} \int_{0}^{\infty}  Q\left(  R-\left(
M\tau\right)  ^{\frac{1}{a}} Y,\left(  M\tau\right)  ^{\frac
{1}{a}}\eta,\tau\right)  \left[  W\left(  Y\right)
-W\left(  Y-\eta\right)  \right]\,d\eta
\end{split}
\]
The inequality is trivially valid for $Y \leq 0$. For $Y\geq 0$, that is $X \leq R$, note that (see (\ref{T2E6}))  
\begin{equation}
Q\left(  X,Z,\tau\right)  \leq   \Omega\left(
Z,\tau\right)  \,,\qquad \int_{R}^{\infty}\Omega\left(  Z,\tau\right)  dZ\leq
C  R^{\gamma-\rho}=\frac{C}{R^{a}}\,.\label{T3E6}%
\end{equation}

Therefore, it is sufficient to prove that
\[
\begin{split}
\frac{1}{\tau}&\int_{0}^{\infty}\frac{1}{\eta^{1+a}}\left[  W\left(
Y\right)  -W\left(  Y-\eta\right)  \right] \,d\eta \\
&\geq\left(  M\tau\right)  ^{\frac
{1}{a}} \int_{0}^{\infty}%
\Omega\left(  \left(  M\tau\right)  ^{\frac{1}{a}}
\eta,\tau\right)  \left[  W\left(  Y-\eta\right)  -W\left(  Y\right)  \right]
d\eta\,.
\end{split}
\]
Using $\theta = (M\tau)^{\frac{1}{a}}$, we rewrite the previous inequality as 
\[
M\int_{0}^{\infty}\frac{d\eta}{\eta^{1+a}}\left[  W\left(  Y\right)
-W\left(  Y{-}\eta\right)  \right]     \geq\theta^{1+a}\int_{0}^{\infty}%
\Omega\left(  \theta\eta,\tau\right)  \left[  W\left(  Y\right)  -W\left(
Y{-}\eta\right)  \right]  d\eta
\]
Due to the scaling properties in  \eqref{T3E6} it is sufficient to  check the inequality
\[
\int_{0}^{\infty}F\left(  \eta\right)  \left[  W\left(  Y\right)
-W\left(  Y{-}\eta\right)  \right]  d\eta\leq M\int_{0}^{\infty}\frac{1}%
{\eta^{1+a}}\left[  W\left(  Y\right)  -W\left(  Y{-}\eta\right)  \right]\,d\eta
\,,\,\,Y>0\,,
\]
where%
\[
\int_{R}^{\infty}F\left(  Y\right)  dY\leq\frac{C}{R^{a}}\,.
\]
To check this we write
$G\left(  \eta\right)  =\int_{\eta}^{\infty}F\left(  \sigma\right)  d\sigma$, such that we are left to show that 
\[
\int_{0}^{\infty}-\frac{dG\left(  \eta\right)  }{d\eta} \left[  W\left(
Y\right)  -W\left(  Y{-}\eta\right)  \right]  d\eta\leq\frac{M}{a}\int
_{0}^{\infty} -\frac{d}{d\eta}\left(  \frac{1}{\eta^{a}}\right)  \left[
W\left(  Y\right)  -W\left(  Y{-}\eta\right)  \right]\,d\eta\,,\,\, Y>0\,,
\]
that is equivalent to
\[
\int_{0}^{\infty}G\left(  \eta\right)  \frac{d}{d\eta}\left[  W\left(
Y\right)  -W\left(  Y{-}\eta \right)  \right]  d\eta  \leq \frac{M}{a}\int
_{0}^{\infty}\frac{1}{\eta^{a}}\frac{d}{d\eta}\left[  W\left(
Y\right)  -W\left(  Y{-}\eta\right)  \right]\,d\eta\,, \,\, Y >0\,,
\]
or
\[
 \int_{0}^{\infty}G\left(  \eta\right)  W^{\prime}\left(  Y{-}\eta\right)  d\eta
 \leq\frac{M}{a}\int_{0}^{\infty}\frac{1}{\eta^{a}}W^{\prime}\left(
Y-\eta\right)\,d\eta\,,\, \,Y >0\,.
\]
Since $G\left(  \eta\right)  \leq\frac{C}{\eta^{a}}$ and $W$ is increasing this inequality is satisfied and the proof is finished.
\end{proof}

\begin{corollary}
Let $\Psi(X,s)$ be the solution of \eqref{T2E3}. Then
 \begin{equation}
\Psi\left(  X,s\right)  \geq W\left(  \frac{R-X}{\left(  M(t - s)\right)
^{\frac{1}{a}}}\right) \qquad \mbox{ for all } s \in [0, t]\,.  \ \label{T4E1}%
\end{equation}%
\end{corollary}

\subsection{Proof of the invariance of (\ref{B2a})}\label{S.lowerbound}

\begin{proposition}\label{P.lowerbound}
Suppose that $h_{0}$ satisfies (\ref{B1a}) and (\ref{B2a}) with $0<\delta < \rho-\gamma$ and
sufficiently large $R_0$ and that $h\left(  \cdot,t\right)  $ is given 
by (\ref{S1}). Then,  $h\left(  \cdot,t\right)  $ satisfies (\ref{B2a}) as well.
\end{proposition}

\begin{proof}
We use  identities (\ref{S8}), (\ref{T2E2}) and (\ref{T4E1}) to conclude
\begin{equation}\label{low1}
\int_0^R h\left(  x,{t}\right)  dx\geq 
e^{-\beta (1{-}\rho) t}  \int h_{0}\left(  x\right)
W\left(  \frac{R-xe^{-\beta{t}}}{\left(  M{t}\right)  ^{\frac{1}{a}%
}}\right)  dx\,.
\end{equation}
We  define $H_{0}$ via $h_{0}\left(  x\right)  =H_{0}'\left(  x\right)$, $H_{0}(0)  =0$ and 
 since (\ref{B2a}) is satisfied for $t=0$ we have
\begin{equation}\label{low2}
H_{0}\left(  x\right)  \geq x^{1-\rho}\left(  1-\frac{R_{0}^{\delta}%
}{x^{\delta}}\right)  _{+}\,.
\end{equation}
Integrating by parts in \eqref{low1}, using the previous estimate for $H_0$ and the fact that $W'(Y)\geq 0$, we find 
\[
\begin{split}
\int_{0}^{R}h\left(  x,{t}\right)  dx&\geq e^{\beta\left(  \rho-1\right)
{t}}\int_{0}^{\infty}H_{0}\left(  x\right)   W^{\prime}\left(
\frac{R-xe^{-\beta{t}}}{\left(  M{t}\right)
^{\frac{1}{a}}}\right) \frac{e^{-\beta {t}}}{\left( 
Mt\right)  ^{\frac{1}{a}}}dx\\
& \geq  e^{\beta\left(  \rho-1\right)
{t}}\int_{0}^{\infty}x^{1-\rho}\left(  1-\frac{R_{0}^{\delta}%
}{x^{\delta}}\right)  _{+} W^{\prime}\left(\frac{R-xe^{-\beta{t}}%
}{\left( Mt\right)  ^{\frac{1}{a}}}\right)
\frac{e^{-\beta{t}}}{\left( M{t}\right)
^{\frac{1}{a}}}dx\,,
\end{split}
\]
which by the change of variables $xe^{-\beta{t}}=y$ turns into 
\[
\int_{0}^{R}h\left(  x,{t}\right)  dx\geq\int_{0}^{\infty}y^{1-\rho
}\left(  1-\frac{R_{0}^{\delta}e^{-\beta\delta{t}}}{y^{\delta}%
}\right)  _{+} W^{\prime}\left(  \frac{R-y}{\left( M{t}\right)  ^{\frac{1}{a}}}\right)  \frac{1}{\left( 
M{t}\right)  ^{\frac{1}{a}}}dy\,.
\]
With the further  change of
variables $Y=\frac{R-y}{\left(M{t}\right)  ^{\frac{1}{a}}%
}$ we obtain
\[
\int_{0}^{R}h\left(  x,{t}\right)  dx\geq\int_0^{\frac{R- R_{0} e^{-\beta{t}}}{\left( M{t}\right)
^{\frac{1}{a}}}}\left(  R-\left(M{t}\right)
^{\frac{1}{a}}Y\right)  ^{1-\rho} \left(  1-\frac{R_{0}^{\delta}e^{-\beta
\delta{t}}}{\left(  R-\left(  M{t}\right)
^{\frac{1}{a}}Y\right)  ^{\delta}}\right)  _{+} W^{\prime}\left(  Y\right)
dY
\]
where we use that $W^{\prime}\left(  Y\right)  =0$ if $Y\leq0$  and  
that the integrand is zero if $y\leq R_{0}e^{-\beta{t}%
},$ that is if  $Y>\frac{R-R_{0}e^{-\beta{t}}}{\left(
M{t}\right)  ^{\frac{1}{a}}}.$ Rearranging the previous inequality and setting
$A=\min \Big( \frac{R-R_0e^{-\beta  t}}{(M t)^{\frac{1}{a}}}, \frac{R}{2 (M t)^{\frac{1}{a}}}\Big)$ we find
\[
\begin{split}
R^{\rho-1} \int_{0}^{R}h(x,{t})  dx
&\geq \int_0^{A} \left(  1-\left(  M{t}\right)  ^{\frac{1}{a}} R^{-1} Y\right)  ^{1-\rho}\\
& \qquad \cdot \left(
1-  \Big ( \frac{R_0}{R}\Big)^{\delta} \Big ( \frac{e^{-\beta  t}}{1- (M t)^{\frac{1}{a}} \frac{Y}{R}}\Big)^{\delta} \right )_+
W^{\prime}\left(  Y\right) \, dY
\end{split}
\]
Note that in $[0,A]$ we have ${M t}^{\frac{1}{a}} \frac{Y}{R} \leq \frac 1 2 $, so that we  can 
 expand the nonlinear terms in $Y$ to obtain
\[
R^{\rho-1} \int_{0}^{R}h(x,{t})  dx
\geq \int_0^A \Big( 1 - \Big ( \frac{R_0}{R} \Big)^{\delta} e^{-\beta \delta  t} \Big)_+ W'(Y)\,dY
 - \frac{C { t}^{\frac{1}{a}}}{R} \int_0^A Y W'(Y)\,dY\,. 
\]
Now recall that $\int_0^{\infty} W'(Y)\,dY=1$ and $W'(Y) \sim C Y^{-(1+a)}$ as $Y \to \infty$ such that 
\[
 \int_0^A W'(Y)\,dY \geq 1 - \frac{C}{A^a} \qquad \mbox{ and } \qquad \int_0^A Y W'(Y)\,dY \leq C A^{1-a}\,.
\]
As a consequence we find 
\begin{equation}\label{low3}
 \begin{split}
  R^{\rho-1} \int_{0}^{R}h(x,{t})  dx&
\geq \Big( 1 - \Big ( \frac{R_0}{R} \Big)^{\delta} e^{-\beta \delta t} \Big)_+ \Big( 1 - \frac{C}{A^a}\Big) - \frac{C  t^{\frac{1}{a}}}{R}
A^{1-a}\\
&\geq \Big( 1 - \Big ( \frac{R_0}{R} \Big)^{\delta}\Big)_+ + \Big ( \frac{R_0}{R} \Big)^{\delta} \frac{\beta \delta  t}{2} 
-  \Big( 1 - \Big ( \frac{R_0}{R} \Big)^{\delta} e^{-\beta \delta  t} \Big) \frac{C}{A^a} 
- \frac{C  t^{\frac{1}{a}}}{R}
A^{1-a}\,.
 \end{split}
\end{equation}
If $A = \frac{R}{2  (Mt)^{\frac{1}{a}}}$, this implies
\[
 R^{\rho-1} \int_{0}^{R}h(x,{t})  dx \geq \Big( 1 - \Big ( \frac{R_0}{R} \Big)^{\delta}\Big)_+ 
+ \Big ( \frac{R_0}{R} \Big)^{\delta} \frac{\beta \delta t}{2} - \frac{C  t}{R^a}\,.
\]
On the other hand, if $A = \frac{R-R_0e^{-\beta  t}}{ (Mt)^{\frac{1}{a}}}$, we have
\[
 \begin{split}
  R^{\rho-1} \int_{0}^{R}h(x,{t})  dx &
 \geq \Big( 1 - \Big ( \frac{R_0}{R} \Big)^{\delta}\Big)_+ 
+ \Big ( \frac{R_0}{R} \Big)^{\delta} \frac{\beta \delta  t}{2} - \frac{C  t}{R^a}
 \frac{\Big(1- \Big( \frac{R_0}{R} e^{-\beta t}\Big)^{\delta}\Big)}{\Big( 1- \frac{R_0}{R} e^{-\beta t}\Big)^a}\\
& \qquad \qquad - \frac{C  t}{R^a} \Big( 1 - \frac{R_0}{R} e^{-\beta  t} \Big)^{1-a}\,.
 \end{split}
\]
Since $a \in (0,1)$ we have
\[
  \Big( 1 - \frac{R_0}{R} e^{-\beta  t} \Big)^{1-a} \leq 1 \qquad \mbox{ and } \qquad
\frac{1- \Big( \frac{R_0}{R} e^{-\beta t}\Big)^{\delta}}{\Big( 1- \frac{R_0}{R} e^{-\beta  t}\Big)^a} \leq C\,.
\]
Thus, it follows in both cases from \eqref{low3} that
\[
  R^{\rho-1} \int_{0}^{R}h(x,{t})  dx
\geq \Big( 1- \Big( \frac{R_0}{R}\Big)^{\delta} \Big)_+ +  t \Big( \Big( \frac{R_0}{R}\Big)^{\delta} \frac{1}{2} \beta \delta  
- \frac{C }{R^{a}}\Big)\,.
\]
If  $\delta < a $ and $R_0\leq R$ is sufficiently large, the second term on the right hand side is nonnegative and it 
follows that $h(\cdot, t)$ satisfies \eqref{B2a}.
\end{proof}

\section{Existence of self-similar solutions }
\subsection{Existence of a weak self-similar solution}

\begin{proposition}\label{P.main}
For any $\gamma \in (0,1)$ there exists a weak stationary solution $h \in \mathcal{Y}$ of \eqref{A2}.
\end{proposition}
\begin{proof}
 
We have proved that the semigroup $S_{\lambda}(t)$ is weakly continuous and leaves the nonempty, convex and compact set
$\mathcal{Y} \subset \mathcal{M}^+([0,\infty))$ invariant. Then it follows by a variant of Tykonov's fixed point theorem
(see Theorem 1.2 in \cite{EMR05}), 
that there exists $h_{\lambda} \in \mathcal{Y}$ that is stationary under the action of $S_{\lambda}(t)$,
that is $h_{\lambda}$ is a stationary mild solution of \eqref{D2}. 
Due to Lemma \ref{L2} the function $h_{\lambda}$ is also a weak stationary solution of \eqref{D2} and we obtain,
 taking test functions $\psi\in C_0^1([0,\infty))$ in \eqref{S4}, that
\[
 \begin{split}
  \int\psi\left(  x\right)  dx\int_{0}^{\infty}dz\frac{K_{\lambda}\left(
x,z\right)  }{z}h_{\lambda}\left(  z\right)  h_{\lambda}\left(  x\right) =
&  \int\psi\left(  x\right)  dx\int_{0}^{x}dy\frac{K_{\lambda}\left(
y,x-y\right)  }{\left(  x-y\right)  }h_{\lambda}\left(  x-y\right)
h_{\lambda}\left(  y\right)  \\
& - \beta\int\partial_{x}\left(  x\psi\right)  h_{\lambda}\left(  x\right)
dx+\beta\rho\int\psi\left(  x\right)  h_{\lambda}\left(  x\right)  dx\,.\\
\end{split}
\]
We can rewrite this as
\[
\begin{split}
  \int \psi\left(  x\right)  &\partial_{x}\left(  \int_{0}^{x}\int
_{x-y}^{\infty}\frac{K_{\lambda}\left(  y,z\right)  }{z}h_{\lambda}\left(
z\right)  h_{\lambda}\left(  y\right) dz\,dy \right)\,dx \\
& =
-\beta\int\partial_{x}\left(  x\psi\right)  h_{\lambda}\left(  x\right)
dx+\beta\rho\int\psi\left(  x\right)  h_{\lambda}\left(  x\right)  dx
\end{split}
\]
whence
\begin{equation}\label{T4E2}
 \begin{split}
  \int \partial_x\psi\left(  x\right)  &\left(  \int_{0}^{x}\int_{x-y}^{\infty
}\frac{K_{\lambda}\left(  y,z\right)  }{z}h_{\lambda}\left(  z\right)
h_{\lambda}\left(  y\right) dz\,dy \right) \,dx\\
&  \quad -\beta\int\partial_{x}\left(  x\psi\right)  h_{\lambda}\left(  x\right)
dx+\beta\rho\int\psi\left(  x\right)  h_{\lambda}\left(  x\right)
dx   =0\,.
\end{split}
\end{equation}

Using now again the compactness of $\mathcal{Y}$ there exists $h \in \mathcal{Y}$ and a subsequence $\lambda \to 0$, such that
 $h_{\lambda}\rightharpoonup h$.
We need to show that $h$ satisfies \eqref{T4E2} with $K$ instead of $K_{\lambda}$. We can easily pass to the limit in the last two
linear terms on the right hand side. 
In order to show convergence of the nonlinear term, note that with  \eqref{dyadic} we find
\begin{equation}\label{proof1}
 \int_{x-y}^{\infty} \frac{h_{\lambda}(z)}{z} \,dz \leq C (x-y)^{-\rho}\qquad \mbox{ and } \qquad
\int_{x{-}y}^{\infty} \frac{h_{\lambda}(z)}{z^{1-\gamma}}\,dz \leq C (x-y)^{\gamma-\rho}
\end{equation}
with a constant that is independent of $\lambda$. Hence we can conclude that \eqref{proof1} is also valid for $h$.
We are going to show that $I[h] \in L^1_{loc}([0,\infty))$ where
\begin{equation}\label{Idef}
 I[h](x)= \int_0^x \int_{x-y}^{\infty} \frac{K(y,z)}{z} h(z) h(y)\,dz\,dy\,.
\end{equation}
Using \eqref{Ass1}, \eqref{proof1}
and \eqref{B1a} and taking an arbitrary $L <\infty$, we obtain
\[
 \begin{split}
  \int_0^L \int_0^x \int_{x-y}^{\infty} \frac{K(y,z)}{z} h(z) h(y)\,dz\,dy\,dx &
\leq  C \int_0^L \int_0^x h(y) \Big( \frac{y^{\gamma}}{(x{-}y)^{\rho}} + \frac{1}{(x{-}y)^{\rho-\gamma}}\Big) \,dy \,dx \\
& \leq C(L) \int_0^L h(y) \int_y^L \frac{1}{(x{-}y)^{\rho}} \,dx \,dy \\
&\leq C(L) \int_0^L h(y)\,dy \leq C(L)\,
 \end{split}
\]
and similarly one finds $\int_{B} I[h](x)\,dx \to 0$ if $|B| \to 0$ which proves the claim.

Next, we need to show that 
\[
(*):= \Big | \int_0^L \int_0^x \int_{x-y}^{\infty} \frac{K_{\lambda}(y,z)}{z} h_{\lambda}(z)h_{\lambda}(y) \,dz\,dy\,dx
\,-\, \int_0^L \int_0^x \int_{x-y}^{\infty} \frac{K(y,z)}{z} h(z)h(y) \,dz\,dy\,dx \Big | \to 0
\]
as $\lambda \to 0$ for any finite $L>0$. 
Now
\begin{equation}\label{proof2}
 \begin{split}
  (*)&\leq \int_0^L \int_0^x \int_{x-y}^{\infty} \Big| \frac{K_{\lambda}(y,z)}{z} - \frac{K(y,z)}{z}\Big| h(z)h(y)\,dz\,dy\,dx\\
& \quad + \Big | \int_0^L \int_0^x \int_{x-y}^{\infty}  \frac{K_{\lambda}(y,z)}{z} \big[ h_{\lambda}(z)
(h_{\lambda}(y) - h(y)) + h(y) (h_{\lambda}(z)-h(z)\big]\,dz\,dy\,dx\Big |\,
 \end{split}
\end{equation}
and, using \eqref{Ass1}, \eqref{D1} and \eqref{proof1}, we find
\[
 \begin{split}
  \int_0^L & \int_0^x \int_{x-y}^{\infty} \Big| \frac{K_{\lambda}(y,z)}{z} - \frac{K(y,z)}{z}\Big| h(z)h(y)\,dz\,dy\,dx\\
& \leq C\int_0^L \int_0^{\min(\lambda,x)} h(y) \big(  y^{\gamma} (x{-}y)^{-\rho} + (x{-}y)^{\gamma-\rho}\big) \,dy\,dx \\
& \leq  C\int_0^{\lambda} h(y)  y^{\gamma} \int_y^{L} (x-y)^{-\rho}\,dx\,dy + C \int_0^{\lambda} h(y) \int_y^L (x{-}y)^{\gamma-\rho}
\,dx\,dy \\
& \leq C(L)  \int_0^{\lambda}h(y)\,dy \\
& \leq C(L) \lambda^{1-\rho} \to 0 \quad \mbox{ as } \lambda \to 0. 
 \end{split}
\]
Using the weak convergence of $h_{\lambda}$ as well as the same bounds and the fact that $\gamma<\rho$,
 we can argue analogously to conclude that also the second
term in \eqref{proof2} converges to zero as $\lambda \to 0$. This implies that $h$ satisfies \eqref{T4E2} with $K_{\lambda}$
 replaced by $K$ and thus it is a weak self-similar solution.
\end{proof}

\subsection{Continuity of self-similar solutions}

\begin{lemma}\label{L.cont}
 The solution $h\in \mathcal{Y}$ from Proposition \ref{P.main} is continuous on $(0,\infty)$.
\end{lemma}
\begin{proof}
Recall that $h$ solves equation \eqref{T4E2} with $K$ instead of $K_{\lambda}$, that is it satisfies
\begin{equation}\label{eq1}
 \partial_x \Big ( I[h] - \beta x h\Big) = \beta (1{-}\rho) h \qquad \mbox{ in } \mathcal{D}'
\end{equation}
 with $I[h]$ as in \eqref{Idef}.
We have already seen in the Proof  of Proposition \ref{P.main} that  $I[h] \in L^1_{loc}([0,\infty))$. 
Then it follows from equation \eqref{eq1}, that also  $xh \in L^1_{loc}[0,\infty)$ and $I[h] - \beta x h$ is a function of bounded 
variation on any compact subset of $[0,\infty)$. Consequently, $I[h] - \beta x h \in L^{\infty}_{loc}[0,\infty)$.

Next, we are going to show that $I[h]$ is locally bounded on $(0,\infty)$. As a consequence of equation
\eqref{eq1} then  also $h$ is locally bounded on $(0,\infty)$.

We are going to use a variant of Young's inequality for convolutions. In its simplest form it says that
if $h \in L^q(\R)$ and $g \in L^p(\R)$, then $h\star g \in L^r(\R)$ with $\frac{1}{q} + \frac{1}{p} = 1 + \frac{1}{r}$
and
\begin{equation}\label{young1}
 \|f \star g\|_{L^r} \leq \| f\|_{L^q} \| g\|_{L^p}\,.
\end{equation}
Examining the proof of \eqref{young1} (see e.g. \cite{LiebLoss01}, pp. 92), we find that it can be easily adapted to show for
\[
 F(x) = \int_0^x h(y) g(x{-}y)\,dy \qquad \mbox{ and } \qquad G(x)= \int_{x/2}^x h(y) g(x{-}y)\,dy
\]
that for $0<L<\infty$
\begin{equation}\label{young2}
 \| F\|_{L^r[0,L]} \leq \| h\|_{L^q[0,L]} \| g\|_{L^p[0,L]}
\end{equation}
and for $0<a<L<\infty$
\begin{equation}\label{young3}
 \| G\|_{L^r[a,L]} \leq \|h\|_{L^q[a/2,L]} \|g\|_{L^p[0,L]}\,.
\end{equation}
For the convenience of the reader we prove \eqref{young3}. To that aim we are going to show that for any $f \in L^{r'}[a,L]$, where
$r'$ is the dual exponent to $r$, we have
\begin{equation}\label{young4}
 \|fG\|_{L^1[a,L]} \leq \|f\|_{L^{r'}[a,L]} \|h\|_{L^q[a/2,L]} \|g\|_{L^p[0,L]}\,,
\end{equation}
from which \eqref{young3} follows by duality.

We define
\[
 \begin{split}
  \alpha(x,y)&= |f(x)|^{r'/q'} |g(x{-}y)|^{p/q'}\\
\beta(x,y)&= |g(x{-}y)|^{p/r} |h(y)|^{q/r}\\
\gamma(x,y)&= |f(x)|^{r'/p'} |h(y)|^{q/p'}
 \end{split}
\]
and note that $\frac{1}{q'}+ \frac{1}{p'}+\frac{1}{r}=1$. Then we can use H\"older's inequality on $\Omega:=
\{ (x,y) \,|\, a \leq x \leq L, x/2 \leq y \leq x\}$, to find
\[
 \|fG\|_{L^1[a,L]} = \int_{\Omega} \alpha \beta \gamma \,dy\,dx \leq \| \alpha \|_{L^{q'}(\Omega)} \| \beta \|_{L^{r}(\Omega)}
\|\gamma\|_{L^{p'}(\Omega)}\,.
\]
Now
\[
 \begin{split}
  \| \alpha \|_{L^{q'}(\Omega)}^{q'}&= \int_a^L \int_{x/2}^x |f(x)|^{r'} |g(x{-}y)|^p\,dy\,dx\\
&= \int_a^L |f(x)|^{r'} \int_0^{x/2} |g(y)|^p\,dy\,dx \leq \|f\|_{L^{r'}[a,L]}^{r'} \|g\|_{L^p[0,L]}^p\,,
 \end{split}
\]
\[
\begin{split}
 \|\beta\|_{L^r(\Omega)}^r &= \int_a^L\int_{x/2}^x |h(y)|^q |g(x{-}y)|^p\,dy\,dx\\
& \leq \int_a^L \int_{a/2}^x |h(y)|^q |g(x{-}y)|^p\,dy\,dx\\
&=\int_{a/2}^L |h(y)|^q \int_y^L |g(x{-}y)|^p\,dx \,dy \leq \|h\|_{L^q[a/2,L]}^q \|g\|_{L^p[0,L]}^p
\end{split}
\]
and
\[
 \|\gamma\|_{L^{p'}(\Omega)} \leq \|f\|_{^{r'}[a,L]}^{r'} \|h\|_{L^q[a/2,L]}^q\,.
\]
Hence, in summary we find
\[
 \|fG\|_{L^1[a,L]} \leq \|f\|_{L^{r'}[a,L]}^{\frac{r'}{q'} + \frac{r'}{p'}}
\|g\|_{L^p[0,L]}^{\frac{p}{q'}+\frac{p}{r}} \|h\|_{L^q[a/2,L]}^{\frac{q}{r}+\frac{q}{p'}}
\]
and since all the exponents are equal to $1$, this proves \eqref{young4}.

We apply now \eqref{young2} with $p<1/\rho$. Then, since $I[h](x) \leq C(L) F(x)$ on $[0,L]$ it follows 
that $I[h] \in L^p[0,L]$ and consequently $xh \in  L^p[0,L]$ and 
$h \in L^p_{loc}(0,\infty)$. 

In the next step we want to iterate this procedure. However, since $h$ can be singular at $x=0$ (in fact, $h(x) =(1{-}\rho) x^{-\rho}$ for 
$K\equiv 0$), we now have to restrict ourselves to compact subsets of $(0,\infty)$. In the following we  consider $[a,L]$
with $a>0$ and $a<L<\infty$, but otherwise arbitrary. Then, using \eqref{young3}, we find
\[
 \begin{split}
  \int_a^L |I[h](x)|^r \,dx & \leq C(L) \int_a^L \Big | \int_0^x \frac{h(y)}{(x{-}y)^{\rho}} \,dy \Big |^r \,dx\\
&\leq C \Big \{ \int_a^L \Big | \int_{x/2}^x \frac{h(y)}{(x{-}y)^{\rho}} \,dy\Big |^r \,dx + \int_a^L \Big | 
\int_0^{x/2} \frac{h(y)}{(x{-}y)^{\rho}} \,dy\Big|^r \,dx \Big\}\\
&\leq C(L) \Big\{ \| h\|_{L^q[\frac{a}{2},L]}^r\| x^{-\rho}\|_{L^p[0,L]}^r + \int_a^L \Big ( \frac{2}{x}\Big)^{\rho r} 
\Big( \frac{x}{2}\Big)^{r(1-\rho)}\,dx \Big\}\\
& \leq C(a,L) \Big( 1+ \| h\|_{L^q[\frac{a}{2},L]}^r\| x^{-\rho}\|_{L^p[0,L]}^r\Big)\,.
 \end{split}
\]

In the first step we take $q=p$ with $p \rho <1$ as above. Then we find $I[h] \in L^r_{loc} (0,L)$ with $
\frac{1}{r} = \frac{2}{p} -1$. Consequently, also $h  \in L^r_{loc} (0,L)$, and we can iterate the procedure until after a finite number
of steps we can take $r=\infty$. This proves that $I[h]$ and $h$ are locally bounded on $(0,\infty)$.

It remains to show that $h$ is continuous. To that aim we are going to show that $I[h]$ is continuous from which the claim follows.
Let $0<x_1<x_2<\infty$. Then
\[
 \begin{split}
  \Big | I[h](x_2)-I[h](x_1)\Big | & \leq 
\Big | \int_{x_1}^{x_2} \int_{x_2{-}y}^{\infty} \frac{K(y,z)}{z} h(z) h(y)\,dz\,dy \Big|
\\
& \qquad + \Big | \int_0^{x_1} \int_{x_1{-}y}^{x_2{-}y} \frac{K(y,z)}{z} h(z)h(y)\,dz\,dy\Big|\\
& \leq C \Big | \int_{x_1}^{x_2} \frac{1}{(x_2{-}y)^{\rho}} \,dy \Big |
+ C \int_0^{x_1} \int_{x_1-y}^{x_2-y} \frac{h(z)}{z}\,dz\,dy\,.
 \end{split}
\]

The first term on the right hand side clearly converges to zero as $x_1 \to x_2$. For the second term  we can estimate
\[
\begin{split}
  \int_0^{x_1} \int_{x_1-y}^{x_2-y} \frac{h(z)}{z}\,dz\,dy &= \int_0^{x_1} \int_{x_1}^{x_2} \frac{h(z{-}y)}{z{-}y} \,dz \,dy
 =
\int_{x_1}^{x_2} \int_0^{x_1} \frac{h(z{-}y)}{z{-}y} \,dy\,dz
\\ &  \leq C \int_{x_1}^{x_2} \frac{1}{(z{-}x_1)^{\rho}}\,dz =C \frac{1}{1{-}\rho} (x_2-x_1)^{1{-}\rho} \to 0 
\end{split}
\]
 as $x_1\to x_2$. This finishes the proof of the Lemma.

\end{proof}

\subsection{Decay behaviour}

\begin{lemma}\label{h.tail}
 The solution $h\in \mathcal{Y}$ from Proposition \ref{P.main} satifies  $h(x) \sim (1{-}\rho) x^{-\rho}$ as $x \to \infty$.
\end{lemma}
\begin{proof} 
Taking (after some approximation procedure) 
 in \eqref{T4E2} with $K$ instead of $K_{\lambda}$ as test function the characteristic function in the interval $[0,R]$  
we obtain
\begin{equation}\label{decay0}
 \int_0^R \int_{R-y}^{\infty} \frac{K(y,z)}{z} h(z)h(y)\,dz\,dy - \beta R h(R) - \beta (\rho{-}1) \int_0^R h\,dx=0.
\end{equation}
Using, as in the proof of Proposition \ref{P.main}, assumption \eqref{Ass1}, \eqref{B1a}  as well as \eqref{proof1}, we find
\begin{equation}\label{decay1}
\begin{split}
 \int_0^R \int_{R-y}^{\infty} \frac{K(y,z)}{z} h(z)h(y)\,dz\,dy &
\leq C R^{\gamma} \int_0^R \frac{h(y)}{(R-y)^{\rho}} \,dy\\
&= C R^{\gamma} \Big( \int_0^{R/2} \frac{h(y)}{(R-y)^{\rho}} \,dy  + \int_{R/2}^R \frac{h(y)}{(R-y)^{\rho}} \,dy \Big)\,.
\end{split}
\end{equation}
We first observe, using \eqref{B1a}, that 
\[
   \int_0^{R/2} \frac{h(y)}{(R-y)^{\rho}}\,dy \leq C R^{-\rho} \int_0^{R/2} h(y)\,dy
\leq C R^{1-2\rho}\,.
\]
Next, we denote the variable in \eqref{decay0} as $x$, divide by $\beta x(R-x)^{\rho}$ and integrate from $R/2$ to $R$.
This gives
\begin{equation}\label{decay2}
 \begin{split}
  \int_{R/2}^R \frac{h(x)}{(R-x)^{\rho}} \,dx &
\leq (1{-}\rho) \int_{R/2}^R \frac{1}{x (R-x)^{\rho}} \int_0^x h(y)\,dy \\
& \quad + C \int_{R/2}^R \frac{x^{\gamma}}{x (R-x)^{\rho}} \int_0^x  \frac{h(y)}{(x-y)^{\rho}} \,dy=: (I) + (II)\,.
 \end{split}
\end{equation}
The first term of the right-hand side of \eqref{decay2} is easily estimated, using \eqref{B1a} and $\rho<1$,  as
\[
 (I) \leq (1{-}\rho) \int_{R/2}^R \frac{1}{x^{\rho} (R-x)^{\rho}} \,dx = (1{-}\rho) R^{1-2\rho} \int_{1/2}^1 \frac{1}{t^{\rho} 
(1{-}t)^{\rho}}\,dt  \leq C R^{1-2\rho}\,.
\]
Furthermore, exchanging the order of integration, we find
\[
 \begin{split}
 (II) &  \leq \frac{C}{R} \int_{R/2}^R \frac{x^{\gamma}}{(R-x)^{\rho}} \int_0^x \frac{h(y)}{(x-y)^{\rho}}\,dy\,dx\\
&\leq C R^{\gamma-1} \Big( \int_0^{R/2} h(y) \int_{R/2}^R \frac{1}{(R-x)^{\rho} (x-y)^{\rho}}\,dx\,dy \\
&\qquad \quad +
\int_{R/2}^R h(y) \int_y^R \frac{1}{(R-x)^{\rho} (x-y)^{\rho}}\,dx\,dy\Big)\,.
 \end{split}
\]
Since we have for $y\in (0,R/2)$ that  
\[
 \int_{R/2}^R \frac{1}{(R-x)^{\rho} (x-y)^{\rho}}\,dx = R^{1-2\rho} \int_{1/2}^1 \frac{1}{(1-t)^{\rho}
(t-y/R)^{\rho}} \,dt \leq C R^{1-2\rho} 
\]
and
\[
 \int_y^R \frac{1}{(R-x)^{\rho} (x-y)^{\rho}}\,dx = R^{1-2\rho} \int_{y/R}^1 \frac{1}{(1-t)^{\rho}
(t-y/R)^{\rho}}\,dt \leq C R^{1-2\rho}\,,
\]
we find, using again \eqref{B1a}, that 
\[
 (II)\leq C R^{-2\rho+\gamma}  \int_0^{R} h(y)\,dy \leq C R^{1-3\rho+\gamma}\,.
\]
In summary, since $\rho>\gamma$ we obtain
for $R>1$ that 
\begin{equation}\label{decay3}
 \int_{R/2}^R \frac{h(x)}{(R-x)^{\rho}}\,dx \leq C\big( R^{1-2\rho}+ R^{1-3\rho+\gamma}\big) \leq C R^{1-2\rho}\,, 
\end{equation}

so that in total we deduce from \eqref{decay1} that
\[
 \int_0^R \int_{R-y}^{\infty} \frac{K(y,z)}{z} h(z)h(y)\,dz\,dy \leq C R^{1-2\rho+\gamma}\,.
\]
Thus, \eqref{decay0} as well as  property \eqref{B2a}, imply that 
\[
\begin{split}
 \Big | \frac{h(R)}{\frac{1{-}\rho}{R} \int_0^R h(x)\,dx } - 1 \Big | &\leq C \frac{R^{-\rho-(\rho{-}\gamma)} }{
 \frac{(1{-}\rho) }{R}\int_0^R h(x)\,dx}\\
& \leq C \frac{R^{-(\rho{-}\gamma)}}{ \Big( 1 - \Big (\frac{R_0}{R}\Big)^{\delta} \Big)_+ } \to 0
\end{split}
\]
as $R \to \infty$. In particular this implies that 
\[
 \Big | \frac{h(R)}{\frac{1{-}\rho}{R} \int_0^R h(x)\,dx }\Big| \leq 2
\]
for sufficiently large $R$. Hence, with some $\omega(R) \to 0$ as $R \to \infty$, we have, using \eqref{B2a}, that
\[
 \begin{split}
  \Big | \frac{h(R)}{(1{-}\rho) R^{-\rho}} - 1\Big| & 
\leq \Big | \frac{h(R)}{(1{-}\rho) R^{-\rho}} - \frac{h(R)}{\frac{1{-}\rho}{R} \int_0^R h(x)\,dx }\Big | + \omega(R)\\
& \leq \frac{h(R)}{(1{-}\rho) R^{-(1{+}\rho)} \int_0^R h(x)\,dx}
\Big( R^{-\rho} - \frac{1}{R} \int_0^R h(x)\,dx \Big) + \omega(R)\\
& \leq \frac{h(R)}{(1{-}\rho) \frac{1}{R} \int_0^R h(x)\,dx} \Big( \frac{R_0}{R}\Big)^{\delta} + \omega(R) \\
& \to 0 \qquad \mbox{ as } R \to \infty\,,
 \end{split}
\]
which finishes the proof of the Lemma.

\end{proof}

\bigskip
{\bf Acknowledgments:} This work was supported by the EPSRC Science and Innovation award to the
Oxford Centre for Nonlinear PDE (EP/E035027/1) and the Hausdorff Center for Mathematics at the University
of Bonn.

\bibliographystyle{amsplain}%
\bibliography{../../coagulation}%

\end{document}